\documentclass[smallextended,final,envcountsect,]{svjour3_JCA}
\smartqed
\usepackage{graphicx}
\usepackage{float}
\usepackage{cite}
\usepackage{amscd} 
\usepackage{verbatim}
\usepackage{amsxtra}
\usepackage{csquotes}
\usepackage{mathtools}
\usepackage{scrextend}
\usepackage{amsmath,etoolbox}
\usepackage{amssymb}
\usepackage{paralist}
\usepackage{graphics} 
\usepackage{epsfig} 
\usepackage{subcaption}
\usepackage{graphicx}  
\usepackage{epstopdf}
\usepackage[colorlinks=true]{hyperref}
\usepackage{centernot}
\usepackage{extarrows}
\usepackage{mathrsfs}
\hypersetup{urlcolor=blue, citecolor=red}
\usepackage{microtype}
\usepackage[mathscr]{euscript}
\usepackage[shortlabels]{enumitem}
\usepackage{lipsum,calc}
\usepackage[linesnumbered,ruled,vlined]{algorithm2e}
\SetKwInput{KwInput}{Input}               
\SetKwInput{KwOutput}{Output}
\SetKwInput{KwWrite}{Write}
\SetKw{KwBy}{by}
\makeatletter
\newenvironment{sqcases}{%
  \matrix@check\sqcases\env@sqcases
}{%
  \endarray\right.%
}
\def\env@sqcases{%
  \let\@ifnextchar\new@ifnextchar
  \left[
  \def\arraystretch{1.2}%
  \array{@{}l@{\quad}l@{}}%
}
\makeatother

\makeatletter
\newcommand{\raisemath}[1]{\mathpalette{\raisem@th{#1}}}
\newcommand{\raisem@th}[3]{\raisebox{#1}{$#2#3$}}
\makeatother

\def\gi{\gamma_i}

\def\f{\longrightarrow}

\def\N{\mathbb{N}}

\def\e{\varepsilon}
\def\l{\lambda}

\def\x{\bar{x}}

\def\y{\bar{y}}
\def\u{\bar{u}}

\def\a{\alpha}
\def\b{\beta}
\def\<{\langle}
\def\>{\rangle}

\def\e{\varepsilon}
\def\R{\mathbb{R}}

\def\inte{\textnormal{int}\,}
\def\clo{\textnormal{cl}\,}

\def\bdry{\textnormal{bdry}\,}

\def\conv{\textnormal{conv}\,}

\def\vol{\textnormal{Vol}}

\def\gk{\gamma_k}

\def\CO{\mathcal{C}}

\def\-{\textnormal{-}}
\def\I{\mathcal{I}}

\def\sp{\hspace{0.015cm}}

\def\U{\mathscr U}

\def\0gk{\scaleto{0,\gk}{3.5pt}}

\binoppenalty=\maxdimen 
\relpenalty=\maxdimen 

\begin{document}

\title{Optimal Control of Sweeping Processes with Nonsmooth Moving Sets: A Numerical Algorithm}
\titlerunning{Optimal Control of Sweeping Processes \dots}

\author{Chadi Nour}

\institute{Chadi Nour \at
             Department of Computer Science and Mathematics \\
             Lebanese American University\\
             Byblos Campus, P.O. Box 36\\
              Byblos, Lebanon\\
              cnour@lau.edu.lb }

\date{Received: date / Accepted: date}

\maketitle

\begin{abstract} This paper extends the numerical method introduced by de Pinho et al. \cite{pinhonum} and later developed in Nour and Zeidan \cite{verachadinum1,verachadinum2} to the nonautonomous case where the nonsmooth sweeping set depends explicitly on time. The moving character of the sweeping set creates substantial additional difficulties, since several geometric constants involved in the analysis may a priori depend on time. To overcome this issue, we establish uniform geometric estimates for the moving constraint sets and their boundaries. These estimates allow us to extend the discrete approximation scheme developed in \cite{pinhonum,verachadinum1,verachadinum2} and prove its convergence toward admissible solutions of the original problem. Several examples illustrating the applicability of the proposed method are also presented.\end{abstract}

\section{Introduction}  This paper is concerned with the development of a numerical algorithm for the following fixed-time Mayer problem
$$ \begin{array}{l} (P)\colon\; \hbox{Minimize}\;g(x(T))\\ \hspace{0.9cm} \hbox{over}\;(x,u)\;\hbox{such that}\;u\in \mathscr{U},\; x\in AC([0,T];\R^n),\, \hbox{and} \\[2pt]  \hspace{0.9cm} (D)  \begin{cases}\dot{x}(t)\in f(x(t),u(t))-N_{C(t)}(x(t)), \;\;\hbox{a.e.}\;t\in[0,T]\\ x(0)=x_0, \end{cases}
 \end{array}$$
where $T>0$ is fixed, $g\colon\R^n\to\R$, $f\colon\R^n\times\R^m\to\R^n$, and $C(t)$ is the intersection of the zero-sublevel sets of a finite family of functions $\psi_i\colon[0,T]\times\R^n\to\R$, $i=1,\dots,r$, namely,
$$
C(t):=\{x\in\R^n:\psi_i(t,x)\le 0,\ i=1,\dots,r\}.
$$
Here, $N_{C(t)}(\cdot)$ denotes the Clarke normal cone mapping to $C(t)$, $x_0\in C(0)$ is fixed, $AC([0,T];\R^n)$ is the set of absolutely continuous functions from $[0,T]$ into $\R^n$, and, for a given nonempty set $U\subset\R^m$, the set of control functions $\mathscr{U}$ is defined by
$$\mathscr{U}:=\left\{ u\colon[0,T]\f\R^m \; \;  \hbox{is  measurable and}\;\; u(t) \in U \;\,\hbox{a.e.}\; t\in [0,T]\right\}. $$
A pair $(x,u)$ is said to be {\it admissible} for $(P)$ if $x\colon[0,T]\to\R^n$ is absolutely continuous, $u\in\mathscr{U}$, and $(x,u)$ satisfies the {\it controlled sweeping process} $(D)$. An admissible pair $(\x,\u)$ is called an optimal solution of $(P)$ if $ g(\x(T))\le g(x(T)) $ for every admissible pair $(x,u)$. In this case, $\x$ and $\u$ are respectively called an {\it optimal trajectory} and an {\it optimal control} of $(P)$.

Sweeping processes were introduced by Moreau in \cite{moreau1,moreau2,moreau3} to describe dynamical systems with unilateral constraints. Since then, they have been extensively studied from both theoretical and applied viewpoints, with applications in mechanics, engineering, economics, and crowd motion problems; see, for instance, \cite{outrata}. In particular, important progress has been achieved in the study of optimal control problems governed by sweeping processes, especially concerning existence results and necessary optimality conditions; see, e.g., \cite{Bettiol,brokate,ccmn,ccmnbis,cmo0,cmo,cmo2,cmonew,verasamara,chhm2,chhm,cmn0,henrion,palladino,pinho,pinhoEr,pinholast,VCpaper,verachadisvaa,verachadijune,verachadijca,verachadi}.

Numerical methods for optimal control problems governed by sweeping processes remain rather limited in the literature, see \cite{outrata,pinhonum,verachadinum1,verachadinum2}. In \cite{pinhonum}, an exponential penalization approach was introduced for the case where the sweeping set $C$ is autonomous,\footnote{When the sweeping set is autonomous, the sweeping process is also referred to as {\it reflected dynamics}.} convex, and smooth, and where the initial condition lies in the interior of $C$. The method transforms the sweeping dynamics into a family of standard control systems that can be solved numerically over piecewise constant controls. This approach was later extended in \cite{verachadinum1} to autonomous nonconvex $\CO^{1,1}$-smooth sweeping sets and to arbitrary initial points $x_0\in C$. More recently, in \cite{verachadinum2}, the method was generalized to autonomous nonsmooth sweeping sets represented as intersections of zero-sublevel sets of finitely many $\CO^{1,1}$ functions. Note that the exponential penalization technique used in \cite{pinhonum} was first introduced in \cite{pinho} in the derivation of a Pontryagin-type maximum principle for controlled sweeping processes and was subsequently developed further in \cite{verasamara,pinholast,pinho22,VCpaper,verachadisvaa,verachadijune,verachadijca,verachadi}.

The purpose of the present paper is to extend this numerical method to the nonautonomous framework in which the sweeping set depends explicitly on time. Although this extension may appear natural at first sight, it creates substantial additional difficulties. Indeed, several geometric constants arising in the analysis may a priori depend on the time variable, and many estimates established in the autonomous case are no longer directly applicable. To overcome this issue, we establish uniform geometric and analytical estimates for the moving sweeping sets and their boundaries. These estimates allow us to extend the approximation scheme developed in \cite{pinhonum,verachadinum1,verachadinum2} and prove its convergence toward admissible solutions of the original problem. The effectiveness of the proposed numerical algorithm is illustrated through two examples involving both smooth and nonsmooth moving sweeping sets.

The paper is organized as follows. In Section~\ref{Hypo}, we introduce the basic notations, definitions, and the hypotheses. Section~\ref{mainresult} contains the approximation scheme, the associated numerical algorithm, and the statement of the main convergence result. In Section~\ref{examples}, we present two numerical examples illustrating the proposed numerical algorithm. Section~\ref{proofsection} is devoted to the proof of the main theorem. Finally, Section~\ref{conclusionsection} contains some concluding remarks.

\section{Basic Notations and Definitions--Hypotheses}\label{Hypo}

\subsection{Basic Notations and Definitions}

Throughout the paper, $\|\cdot\|$ and $\<\cdot,\cdot\>$ stand for the Euclidean norm and the standard inner product on $\R^n$, respectively. The symbols $B$ and $\bar B$ denote the open and closed unit balls of $\R^n$. More generally, for $x\in\R^n$ and $\rho\geq0$, we write $B_{\rho}(x)$ and $\bar{B}_{\rho}(x)$ for the open and closed balls centered at $x$ with radius $\rho$, respectively. For $a,b\in\R^n$, the closed and open segments joining $a$ and $b$ are denoted by $[a,b]$ and $]a,b[$, respectively. Given a subset $S$ of $\R^n$, the symbols $\inte S$, $\bdry S$, $\clo S$, $\conv S$, and $S^c$ refer respectively to the interior, boundary, closure, convex hull, and complement of $S$. The space of absolutely continuous mappings from $[a,b]$ into $\R^n$ is denoted by $AC([a,b];\R^n)$. A function $f\colon\R^n\to\R$ is said to be of class $\CO^{1,1}$ whenever it is Fr\'echet differentiable and its derivative is locally Lipschitz continuous. Finally, a mapping $\pi\colon A\to B$ is called bi-Lipschitz if it is a Lipschitz bijection from $A$ onto $B$ whose inverse is also Lipschitz.

We now recall several notions from {\it Nonsmooth Analysis} and {\it Geometry}; see, for instance, the monographs \cite{brudnyi,clarkeold,clsw,delfour,mordubook,rockwet,ThibaultBook}. Let $S$ be a nonempty closed subset of $\R^n$ and let $s\in S$. The proximal, Mordukhovich (or limiting), and Clarke normal cones to $S$ at $s$ are denoted by $N^P_S(s)$, $N_S^L(s)$, and $N_S(s)$, respectively. The Clarke tangent cone to $S$ at $s$ is denoted by $T_S(s)$. Given $\rho>0$, the set $S$ is said to be $\rho$-prox-regular if, for every $s\in S$ and every unit vector $\zeta\in N_S^P(s)$, the inequality $\langle \zeta,x-s\rangle \leq \frac{1}{2\rho}\|x-s\|^2$ holds for all $x\in S$. This inequality is usually referred to as the proximal normal inequality. Finally, $S$ is said to be quasiconvex if there exists a constant $c\geq0$ such that any two points $s_1,s_2\in S$ can be connected by a polygonal curve $\gamma$ contained in $S$ and satisfying $l(\gamma)\leq c\,\|s_1-s_2\|,$ where $l(\gamma)$ denotes the length of $\gamma$.

\subsection{Hypotheses}

Let $C_{[0,T]}$ be the set defined by $$C_{[0,T]}:=\bigcup_{t\in[0,T]} C(t).$$
We assume throughout this paper that the data of $(P)$ satisfy the following hypotheses. \begin{enumerate}[label=\textbf{H\arabic*}:, leftmargin=\widthof{[H4]}+\labelsep]
\item $f$ is continuous on $C_{[0,T]}\times U$; and there exists $M>0$ such that for every $u\in U$, the mapping $x\mapsto f (x, u)$ is $M$-Lipschitz on $C_{[0,T]}$; and $\|f(x,u)\| \leq M$ for all $(x,u)\in C_{[0,T]} \times U.$
\item $f(x,U)$ is convex for all $x\in\R^n$, and $U$ is compact.
\item $C(t)\neq\emptyset$ for all $t\in[0,T]$ is given by
\[C(t):=\bigcap_{i=1}^r C_i(t),\;\hbox{where}\;C_i(t):=\{x\in\R^n : \psi_i(t,x)\leq0\},\;\hbox{and}\] 
$(\psi_i)_{1\le i\le r}$ is a family of continuous functions such that $\psi_i(t,\cdot)$ is $\CO^{1,1}$ for every $t\in[0,T]$. Moreover, for every $i=1,\ldots,r$ and every $t\in [0,T]$,  $\operatorname{bdry} C_i(t)$ is connected for $n>1$, and $C_i(t)$ is convex for $n=1$.
Additionally, there exist constants $M_t>0$, $M_\psi>0$ and $R>0$ such that, for every
$i=1,\ldots,r$ and every $t\in[0,T]$, $C_i(t)\subset B(0;R)$, and, for all
$t,s\in[0,T]$ and all $x,y\in B(0;2R)$,
\[|\psi_i(t,x)-\psi_i(s,x)|\leq M_t|t-s|,\;\hbox{and}\]
\begin{equation}
\|\nabla_x\psi_i(t,x)-\nabla_x\psi_i(t,y)\|
\le
2M_\psi\|x-y\|.\label{H3ineq1}
\end{equation} Finally, there exists $\eta>0$ such that
 \begin{equation}\label{H3ineq2} \left\|\sum_{i\in\I^0_c(t)}\lambda_i\nabla_x\psi_i(t,c)\right\|>2\eta, \;\;  \forall c\in \{x\in \R^n: \I^0_x(t)\not=\emptyset\},\end{equation} where $\I^0_x(t):=\{i\in\{1,\dots,r\} : \psi_i(t,x)=0\}$ and $(\l_i)_{i\in\I^0_c(t)}$ is any sequence of nonnegative numbers satisfying $\sum_{i\in\I^0_c(t)}\l_i=1$. 
\item $g$ is $L_g$-Lipschitz on $C(T)$.
\end{enumerate}
\begin{remark}\label{rem1} By (H3), for every $t\in[0,T]$ and every $c\in\bdry C_i(t)$, taking $\lambda_i=1$ in \eqref{H3ineq2} yields
$ \|\nabla_x\psi_i(t,c)\|>2\eta$. Hence, for every $i\in\{1,\dots,r\}$, $t\in[0,T]$, and every $c\in\bdry C_i$, we have $ \|\nabla_x\psi_i(t,c)\|>2\eta$.
\end{remark}
\begin{remark}\label{Mt} Throughout the paper, and without further mention, the constant $M$ appearing in {\rm(H1)} is replaced by $M+\frac{M_t}{\eta}$. For simplicity, and to keep the constants consistent with those used in \cite{verachadinum2}, we continue to denote this enlarged constant by $M$.
\end{remark}

\section{Main Result--Numerical Algorithm}\label{mainresult}
We denote by $\bar{M}_\psi$ a common upper bound of the family $(\|\nabla_x\psi_i(t,x)\|)_{i=1}^{r}$ over $[0,T]\times \overline{B}(0;R)$, chosen such that $\bar{M}_\psi\ge 2\eta$. Moreover, increasing $M_\psi$ if necessary, we assume that $M_\psi\ge \frac{\eta}{R}$. We consider $(\gk)_k$ a sequence of positive real numbers satisfying
$$\gk\geq\frac{2Me}{\eta}>\frac{2M}{\eta}\;\,\hbox{for all } k\in\N,$$
and such that $\gk\to\infty$ as $k\to\infty$. Correspondingly, we define the real sequences $(\a_k)_k$ and $(\sigma_k)_k$ by
\begin{equation}\label{sigmadef}
\a_k:=\frac{\ln \left(\frac{\eta\gk}{2M}\right)}{\gk}
\;\,\hbox{and}\;\,
\sigma_k:=\frac{rM{\psi}}{2\eta^2}\left(\frac{\ln(r)}{\gk}+\a_k\right),\;\,k\in\N.
\end{equation}
By construction, the sequence $(\a_k)_k$ satisfies $ \gk e^{-\a_k\gk}=\frac{2M}{\eta}$, with $\a_k>0$ for all $k\in\N$ and $\a_k\searrow 0$. Likewise, $(\sigma_k)_k$ satisfies $\sigma_k>0$ for all $k\in\N$ and $\sigma_k\searrow 0$. We now introduce the functions $\psi_{\gk}\colon[0,T]\times\R^n\to\R$ defined by  \begin{equation*}\label{psigkdef} \psi_{\gk}(t,x):=\frac{1}{\gk}\ln\Bigg(\sum_{i=1}^{r}e^{\gk \psi_i(t,x)}\Bigg),\;\;\forall (t,x)\in[0,T]\times \R^n. \end{equation*}
For every $k\in\N$, the function $\psi_{\gk}$ is continuous on $[0,T]\times\R^n$, and, for every $t\in[0,T]$, the function $\psi_{\gk}(t,\cdot)\in C^{1,1}(\R^n;\R)$. Moreover, the sequence $(\psi_{\gk})_k$ is monotonically nonincreasing in $k$, and converges uniformly to $\psi$ on $[0,T]\times\R^n$. In addition, for all $k\in\N$, for all $i=1,\dots,r$, and for all $(t,x)\in[0,T]\times\R^n$, we have
\begin{equation*}\label{17} \psi_i(t,x)\le\psi(t,x)\le\psi_{\gk}(t,x)\le\psi(t,x)+\frac{\ln(r)}{\gk}.\end{equation*}
A direct computation yields \begin{equation}\label{gradpsigk}
\nabla_x\psi_{\gk}(t,x)= \frac{\sum\limits_{i=1}^r e^{\gk\psi_i(t,x)} \nabla_x\psi_i(t,x)}{\sum\limits_{i=1}^re^{\gk\psi_i(t,x)}},\;\;\forall (t,x)\in[0,T]\times\R^n.
\end{equation}
In order to define the approximating dynamics, let $c\in \bdry C(0)$ and denotes by $d_c$ the nonzero vector $d_c:=\sum_{j\in\I^0_c}v_j(c) $, where for $j=1,\dots,r$, $v_j(c)$ is the unique projection of $-\nabla_x\psi_j(0,c)$ onto the Clarke tangent cone $T_{C(0)}(c)$. We then define the {\it approximation} dynamic $(D_{\gk})$ of $(D)$ by  
\begin{equation}\label{Dgk1}\begin{cases} \dot{x}(t)=f(x(t),u(t))-\sum\limits_{i=1}^{r} \gk e^{\gk\psi_i(t,x(t))} \nabla_x\psi_i(t,x(t)),\;\textnormal{a.e.}\; t\in[0,T],\\ x(0)=x^k_{0}, \end{cases}
 \end{equation}
where the sequence $\left(x_0^k\right)_k$ associated with the initial point $x_0$ of $(P)$ is defined by
\begin{equation} \label{initial} x_0^k:= 
     \begin{cases}x_0,\;\forall k\in\N, &\;\hbox{if}\;x_0\in\inte {C(0)},\vspace{0.2cm}\\ \displaystyle {x_0}+\sigma_k \frac{d_{x_0}}{\|d_{x_0}\|},\;\forall k\in\N,&\;\hbox{if}\; x_0\in\bdry C(0).\end{cases} \end{equation}
Using \eqref{gradpsigk}, the dynamic $(D_{\gk})$ can equivalently be rewritten in terms of $\psi_{\gk}$ as follows:
 \begin{equation*}\begin{cases} \dot{x}(t)=f(x(t),u(t))-\gk e^{\gk\psi_{\gk}(t,x(t))} \nabla_x\psi_{\gk}(t,x(t)),\;\textnormal{a.e.}\; t\in[0,T],\\ x(0)=x_0^{k}. \end{cases} \end{equation*}
This leads naturally to the approximating optimal control problem $(P_{\gk})$, obtained from $(P)$ by replacing the original dynamics $(D)$ with its approximation $(D_{\gk})$, namely, $$ \begin{array}{l} (P_{\gk})\colon\; \hbox{Minimize}\;g(x(T))\\ \hspace{1.17cm} \hbox{over}\;(x,u)\;\hbox{such that}\;u\in \mathscr{U}, x\in AC([0,T];\R^n), \;\hbox{and} \\[2pt]  \hspace{0.7cm} ({D}_{\gk}) \begin{cases} \dot{x}(t)=f(x(t),u(t))-\gk e^{\gk\psi_{\gk}(t,x(t))} \nabla_x\psi_{\gk}(t,x(t)),\,\,\textnormal{a.e.}\; t\in[0,T],\\ x(0)=x_0^{k}.  \end{cases}
    \end{array}$$
For $N\in\N$, let $h:=\frac{T}{N}$. We further denote by $\U^N$ the set of piecewise constant control functions $u^N$ on $[0,T]$ with step size $h$, namely, $$ u^N(t)\equiv u_j^N\;\,\hbox{for } t\in[(j-1)h,jh),\;\,j=1,\dots,N, $$
where $\big(u_1^N,\dots,u_N^N\big)\in U^N:=\overbrace{U\times\cdots\times U}^{N}$. In this framework, $x_{\gk}^N$ denotes the solution of $(D_{\gk})$ associated with the control $u^N\in\U^N$. We then introduce the problem $\big(P_{\gk}^N\big)$ obtained from $(P_{\gk})$ by restricting the controls to the class $\U^N$, that is, 
$$ \begin{array}{l} \big(P_{\gk}^N\big)\colon\; \hbox{Minimize}\;g(x(T))\\[1pt] \hspace{1.22cm} \hbox{over}\;(x,u^N)\in AC([0,T];\R^n)\times \mathscr{U}^N\;\hbox{such that for} \;j=1,\dots,N,\\[1pt] \hspace{1.22cm}
 x(\cdot):=x^j(\cdot)\;\hbox{on}\,[(j-1)h, jh]\;\hbox{and}\; x^j \;\hbox{satisfies on}\;  [(j-1)h, jh)\\[3pt]  \hspace{1.2cm} \begin{cases} \dot{x}^j(t)=f(x^j(t),u^N_j)-\gk e^{\gk\psi_{\gk}(t,x^j(t))} \nabla_x\psi_{\gk}(t,x^j(t)), \\[7pt]  x^j((j-1)h)=\begin{sqcases} x_0^{k} & \hbox{if}\;j=1,\\[2pt] x^{j-1}((j-1)h) & \hbox{if}\;j\geq 2. \end{sqcases} \end{cases}
    \end{array}$$
    Since $U^N$ is compact, the problem $\big(P_{\gk}^N\big)$ admits an optimal solution. Let $\left(\hat{x}_{\gk}^N,\hat{u}_{\gk}^N\right)$ be an optimal solution of $\big(P_{\gk}^N\big)$, and denote by $\tilde{x}^N_{\gk}$ the solution of the original sweeping process $(D)$ corresponding to the control $\hat{u}_{\gk}^N$.  
    
Based on the above notation, we now state the following theorem, which constitutes the main result of this paper. Roughly speaking, it shows that $\big(P_{\gk}^N\big)$ approximates $(P)$ as $N\to\infty$ and $k\to\infty$. This extends \cite[Theorem 3.1]{verachadinum2} to the more general setting where the sweeping set depends explicitly on time. Although the statement below formally coincides with the autonomous case treated in \cite{verachadinum2}, its proof (see Section \ref{proofsection}) is substantially more involved due to the explicit time dependence of the sweeping set.

\begin{theorem} \label{th1}  Let $(\x,\u)$ be an optimal solution for $(P)$. For any $\varepsilon >0$, there exist  $k_{\e}\in\N$, and  $N_{\e}\in \N$ such that,  for $k\ge k_{\e}$ and $N\ge N_{\e}$,
  \begin{eqnarray*} &&g(\x(T))\leq g(\tilde{x}^N_{\gk}(T))\leq g(\x(T))+\e,\;\;\hbox{and} \\ && g(\x(T))-\frac{\e}{3}\leq g(\hat{x}^N_{\gk}(T))\leq g(\x(T))+\frac{2\e}{3}.\end{eqnarray*}
Moreover, there exists an optimal solution $(\hat{x},\hat{u})$ of $(P)$ such that, up to a subsequence, both sequences $\hat{x}^N_{\gk}$ and $\tilde{x}^N_{\gk}$ converge uniformly to $\hat{x}$ as $N\f\infty$ and $k\f\infty$.
\end{theorem}

As a direct consequence of Theorem \ref{th1}, we obtain Algorithm \ref{alg1}, which provides a numerical procedure for approximating both optimal trajectories and optimal values of the controlled sweeping process problem $(P)$.
\begin{algorithm}[htb]
\SetAlgoLined
\KwInput{Positive integer $N$, numbers $\e>0$, $\gamma>0$ and $\delta>0$}\vspace{0.04cm}
\KwOutput{$\blacktriangleright$ Approximating minimum value of $(P)$\newline $\blacktriangleright$ Numerical optimal trajectory of $(P)$} \vspace{0.06cm}
 Initialization\;\vspace{0.04cm}
  $k\gets 1$;\vspace{0.04cm}\\
 \For{$i\gets 0$ \KwTo $1$ \KwBy $1$}{\vspace{0.04cm}
 $\gamma_{i}\gets\gamma+i\sp\delta$;\vspace{0.04cm}\\
 compute $x_0^i$;\vspace{0.04cm}\\
 compute a solution $\hat{x}^N_{\gi}$ of $\big(P_{\gi}^N\big)$ for $\gamma_i$ and $x_0^i$;\vspace{0.04cm}\\
$g_i\gets g\big(\hat{x}^N_{\gi}(T)\big)$;\vspace{0.04cm}\\
}\vspace{0.03cm}
\While{$|g_{k}-g_{k-1}|> \e$}{
  \vspace{0.08cm} $k\gets k+1$;\vspace{0.04cm}\\
  $\gamma_k\gets \gamma_{k-1}+\delta$;\vspace{0.04cm}\\
compute $x_0^k$;\vspace{0.04cm}\\
compute a solution $\hat{x}^N_{\gk}$ of $\big(P_{\gk}^N\big)$ for $\gamma_k$ and $x_0^k$;\vspace{0.04cm}\\
$g_k\gets g\big(\hat{x}^N_{\gk}(T)\big)$;\vspace{0.04cm}\\
 }\vspace{0.03cm}
\Return $g_k$;
\caption{Numerical Method for solving $(P)$}
\label{alg1}
\end{algorithm}

\section{Examples and Numerical Simulations} \label{examples}

In this section, we illustrate the applicability of the proposed numerical algorithm on two optimal control problems governed by nonautonomous sweeping processes. The examples are chosen in order to highlight different qualitative behaviors of the optimal trajectories, including interior motion, contact with the moving boundary, constrained motion along the boundary of the sweeping set, as well as situations where the initial condition lies either in the interior or on the boundary of the sweeping set.

For each example, we compare the exact optimal solution, when available, with the numerical trajectories generated by the approximation method developed in the present paper. The numerical computations are performed using a Nelder-Mead optimization procedure coupled with a fourth-order Runge-Kutta discretization of the penalized dynamics.

\begin{example}  \label{Example1} We consider the following data for the problem $(P)$:
\begin{enumerate}[$\bullet$]
\item $n=1$ and $m=1$, and the perturbation mapping $f\colon\R\times\R\f\R$ is defined by $f(x,u):=u$.
\item $r=1$, the terminal time is $T:=3$, and the function $\psi_1\colon[0,3]\times \R\f\R$, denoted now by $\psi$, is defined by $$\psi(t,x)=(x+1)\left(x-\frac34-\frac{t^2}{4}\right).$$
Hence, for every $t\in[0,3]$, the sweeping set $C(t)=C_1(t)=\left[-1,\frac34+\frac{t^2}{4}\right]$.
\item The objective function $g\colon\R\f\R$  is defined  by $g(x):=-x$.
\item The control set is $U:=[-1,1]$ and the initial point is $x_0:=0\in\inte C(0)$.
\end{enumerate}
One can easily verify that our hypotheses are satisfied with $M=28$, $R=4$, $M_{\psi}=1$, $M_t=\frac{27}{2}$, $\eta=\frac{1}{2}$, and $L_g=1$. Observe that for every $(t,x)\in[0,3]\times\R$, the normal cone to $C(t)$ is given by $$N_{C(t)}(x)=\begin{cases} \{0\}, & \text{if } x\in\left]-1,\frac34+\frac{t^2}{4}\right[,\\ (-\infty,0], & \text{if } x=-1,\\ [0,+\infty), & \text{if } x=\frac34+\frac{t^2}{4}. \end{cases}$$ Consequently, the sweeping dynamics associated with $(P)$ takes the form $$\dot x(t)\in u(t)-N_{C(t)}(x(t)),\;\;u(t)\in[-1,1],\;\;\text{for a.e. } t\in[0,3].$$
Since $g(x)=-x$, minimizing $g(x(3))$ is equivalent to maximizing the terminal value $x(3)$. Therefore, the optimal strategy is to push the trajectory as far as possible to the right, namely by taking
\[ \bar u(t)=1,\;\;\forall t\in[0,3]. \]
Starting from $x_0=0$, the free motion is given by $x(t)=t$. This trajectory reaches the upper boundary $b(t):=\frac34+\frac{t^2}{4}$ at the first time $\tau$ satisfying $\tau=b(\tau)$. Since $\tau=\frac34+\frac{\tau^2}{4}$ is equivalent to $\tau^2-4\tau+3=0$, we obtain $\tau=1$. Hence, for $0\le t\le1$, the optimal trajectory is $\bar x(t)=t$. After reaching the upper boundary at time $t=1$, the trajectory may either remain on the boundary or return to the interior of the moving set. In order to remain on the upper boundary, the trajectory must satisfy \[ \bar x(t)=b(t)=\frac34+\frac{t^2}{4},\;\,\hbox{and hence}\;\,\dot{\bar x}(t)=b'(t)=\frac t2. \]
Since $\bar u(t)=1$, the dynamics on the boundary can be written as
\[\dot{\bar x}(t)=1-\xi(t)\;\,\hbox{with}\;\,\xi(t)\ge0,\;\,\hbox{thus}\;\,\xi(t)=1-\frac t2.\] This is nonnegative precisely for $t\le2$. If the trajectory were to leave the boundary immediately after $t=1$, then it would move with free velocity $1$, which would force it to exit the moving set since $b'(t)=\frac t2<1$ for $1<t<2$. Therefore, the trajectory remains on the upper boundary for $1\le t\le2$. At $t=2$, we have $b'(2)=1$, and for $t>2$, the upper boundary moves faster than the free velocity allowed by the control $u=1$. Therefore, the trajectory leaves the boundary and moves again in the interior with velocity $1$. Since $\bar x(2)=b(2)=\frac34+\frac{4}{4}=\frac74$, we get, for $2\le t\le3$, that $\bar x(t)=\frac74+(t-2)$. Consequently, the optimal trajectory is \[ \bar x(t)=\begin{cases} t, & 0\le t\le1,\\[2mm] \dfrac34+\dfrac{t^2}{4}, & 1\le t\le2,\\[2mm] \dfrac74+(t-2), & 2\le t\le3. \end{cases}\] In particular, $\bar x(3)=\frac74+1=\frac{11}{4}$. Therefore, \[\min(P)=g(\bar x(3))=-\frac{11}{4}=-2.75.\]

We next illustrate the performance of Algorithm \ref{alg1} on Example \ref{Example1}. More precisely, we numerically compute approximations of the minimum value and of an optimal trajectory associated with the discrete penalized problem $\big(P_{\gk}^N\big)$. These numerical approximations are then compared with the exact minimum value and the exact optimal trajectory of $(P)$ derived above. The resulting simulations show an excellent agreement between the theoretical and numerical solutions, thereby providing a strong numerical illustration of the convergence result established in Theorem \ref{th1}. 
\begin{figure}[t!]
\centering
\includegraphics[width=70mm]{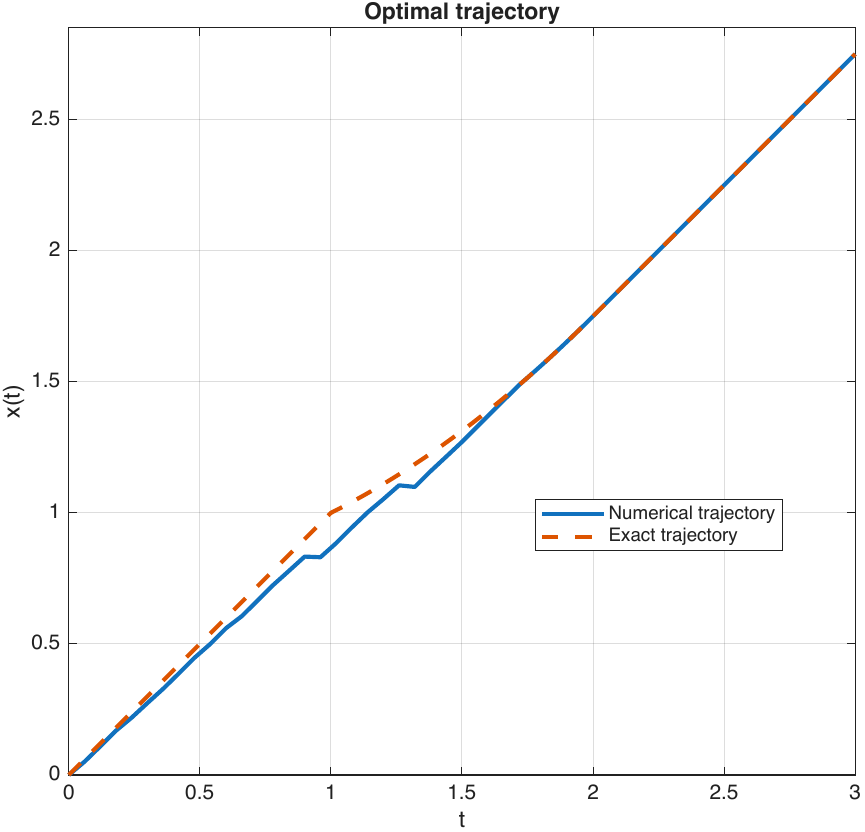}
\caption{\label{Fig1} Numerical vs exact optimal trajectory}
\end{figure}
We choose $N=50$, $\e=10^{-5}$, $\gamma=310$,\footnote{Since $\frac{2Me}{\eta}\approx 304.447$.} and $\delta=10$, and apply Algorithm \ref{alg1} in order to numerically compute an approximation of both the minimum value and an optimal trajectory associated with $(P)$. Since the initial point lies in the interior of $C(0)$, the sequence $(x_0^k)_k$ reduces here to the constant sequence $x_0$.  Starting from $\gamma_1=310$, the algorithm reached the prescribed tolerance $\e$ after $187$ iterations, corresponding to the value $\gamma_{187}=2170$. The obtained numerical cost is $$ g_{187}=g(\hat{x}_{\gamma_{187}}^N(3))=-2.747601272285,$$
which is in excellent agreement with the exact minimum value of $(P)$ computed above. The total running time of the algorithm was $50.394$ seconds.\footnote{Machine: MacBook Pro, Apple M3 chip, 16GB Unified Memory.} Moreover, Fig.~\ref{Fig1} shows that the computed numerical optimal trajectory $\hat{x}_{\gamma_k}^N$ accurately reproduces the qualitative behavior of the exact optimal trajectory $\bar x$, including the interior motion, the contact with the moving boundary, the constrained motion along the boundary, and the return to the interior of the sweeping set, while Fig.~\ref{Fig2} illustrates the very rapid convergence of the numerical minimum values toward the exact minimum value $-\frac{11}{4}$ as the penalization parameter $\gamma_k$ increases. Altogether, these numerical results provide a strong illustration of the convergence result established in Theorem \ref{th1}.
\begin{figure}[h!]
\centering
\includegraphics[width=70mm]{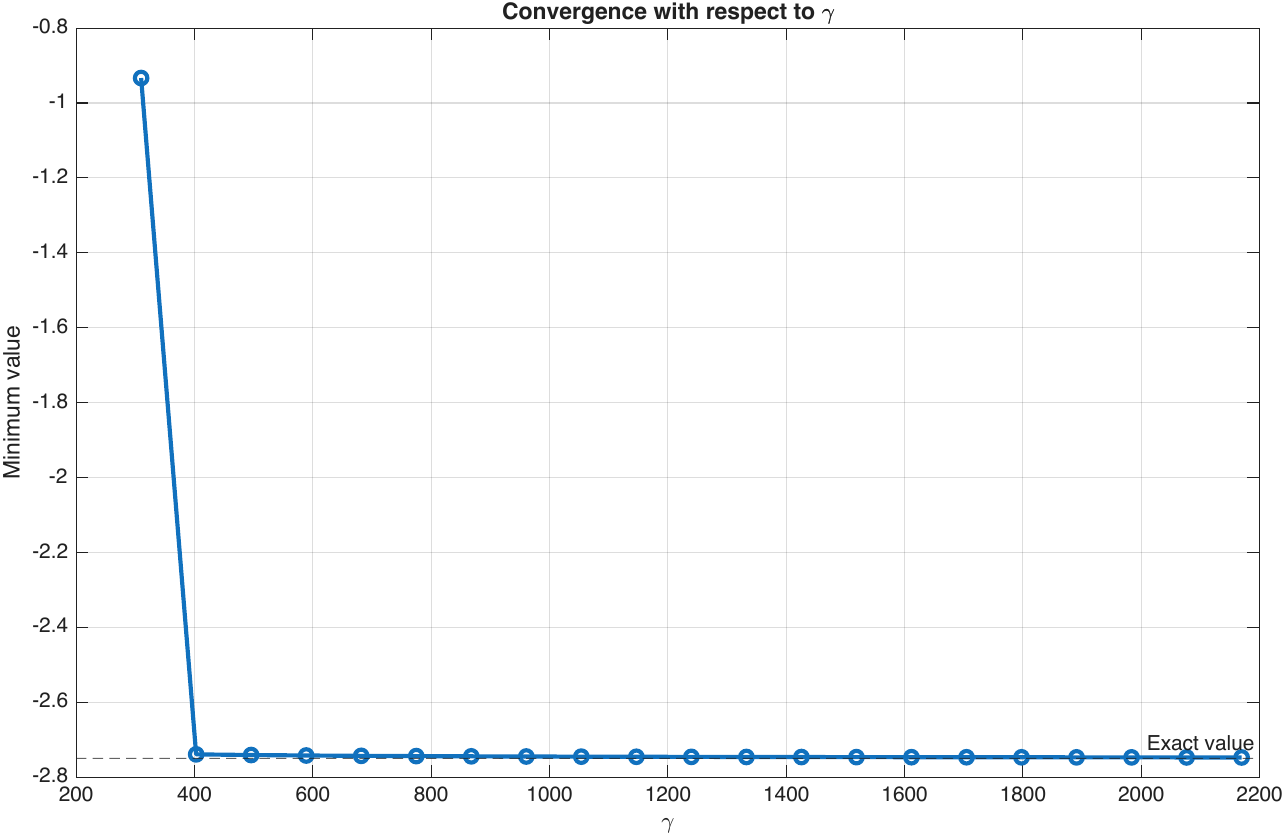}
\caption{\label{Fig2} Convergence with respect to $\gk$}
\end{figure}
\end{example}

\begin{example}\label{Example2}
We consider the following data for the problem $(P)$:
\begin{enumerate}[$\bullet$]

\item $n=3$ and $m=1$, and the perturbation mapping $
f\colon\R^3\times\R\to\R^3$
is defined by $f(x,u):=(0,u,0)$.

\item $r=2$, the terminal time is $T=t_1+\frac{3\sqrt2-1}{2}$ with $t_1=3\ln(1+\sqrt2)$, and the moving sweeping set is given by $C(t):=C_1(t)\cap C_2(t),$ where for $\rho(t):=\sqrt{16+q(t)^2}$ and $q(t)=
\begin{cases}
3, & 0\le t\le t_1,\\
3+(t-t_1), & t_1\le t\le T,
\end{cases}$
\begin{enumerate}[$\bullet$]
\item $C_1(t):=
\left\{
(x_1,x_2,x_3)\in\R^3:
(x_1-4)^2+x_2^2+x_3^2\le \rho(t)^2
\right\}$, and
\item $C_2(t):=
\left\{
(x_1,x_2,x_3)\in\R^3:
(x_1+4)^2+x_2^2+x_3^2\le \rho(t)^2
\right\}$.
\end{enumerate}
Hence, $\psi_1, \psi_2\colon[0,T]\times\R^3\f\R$ are defined by: 
\begin{enumerate}[$\bullet$]
\item $\psi_1(t,x_1,x_2,x_3):=(x_1-4)^2+x_2^2+x_3^2-\rho(t)^2$.
\item $\psi_2(t,x_1,x_2,x_3):=(x_1+4)^2+x_2^2+x_3^2- \rho(t)^2$.
\end{enumerate}
Note that $C(t)$ is a nonsmooth and convex set.
\item The objective function $g\colon\R^3\f\R$  is defined  by $$g(x_1,x_2,x_3)=x_1^2+\left(x_2-\frac12\right)^2+
\left(x_3-\frac{3\sqrt2}{2}\right)^2.$$
\item The control set is $U:=[-1,1]$ and the initial point is $x_0:=(0,0,3)$. Observe that the nonsmooth part of the boundary of $C(t)$ is given by
\[
\Gamma(t):=
\Big\{
(x_1,x_2,x_3)\in\R^3:
x_1=0,\;
x_2^2+x_3^2=q(t)^2
\Big\}.
\]
Hence, $x_0=(0,0,3)\in\Gamma(0),$ that is, the initial point lies on the nonsmooth part of the boundary of $C(0)$,
see Fig. \ref{Fig3} for $C(0)$.
\end{enumerate}
\begin{figure}[h!]
\centering
\includegraphics[width=67mm]{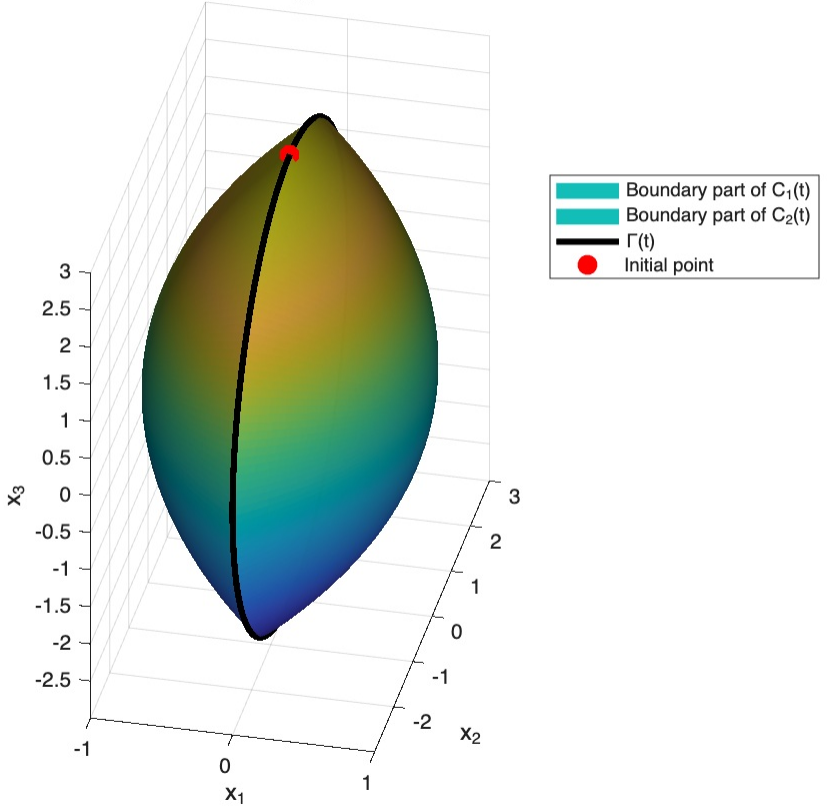}
\caption{\label{Fig3} $C(0)$, $\Gamma(0)$, and $x_0$}
\end{figure}
One can easily verify that our hypotheses are satisfied with $M=\frac{7+3\sqrt{2}}{2}$, $R=11$, $M_\psi=1$, $M_t=5+3\sqrt{2}$, $\eta=2$, and $L_g=17$. Moreover, a simple calculation leads to $d_{x_0}=(0,0,-1)$, and hence from \eqref{sigmadef} and \eqref{initial}, $$\,x_0^k=(0,0,3)+
\frac{1}{4\gamma_k}
\ln\!\left(\frac{4\gamma_k}{7+3\sqrt2}\right)(0,0,-1),\;\,\forall k.$$
By arguments similar to those developed in Example \ref{Example1}, one verifies that the optimal control is given by
\[
\bar u(t)=
\begin{cases}
1, & 0\le t\le t_1,\\
-1, & t_1\le t\le T,
\end{cases}
\]
and that the corresponding optimal trajectory is
\[
\bar x(t)=
\begin{cases}
\bigl(0,3\sin\theta(t),3\cos\theta(t)\bigr),
& 0\le t\le t_1,\\[1mm]
\left(
0,\dfrac{3\sqrt2}{2}-(t-t_1),\dfrac{3\sqrt2}{2}
\right),
& t_1\le t\le T,
\end{cases}
\]
where $
\theta(t)=2\arctan(e^{t/3})-\frac{\pi}{2}$. Thus, the trajectory first evolves along the nonsmooth part $\Gamma(t)$ of the boundary of the moving set, and then enters the interior of $C(t)$ and evolves freely until the terminal time. In particular, $\bar x(T)=\left(0,\frac12,\frac{3\sqrt2}{2}\right),$ and therefore
\[
\min(P)=0.
\]
\begin{figure}[t!]
\centering
\includegraphics[width=95mm]{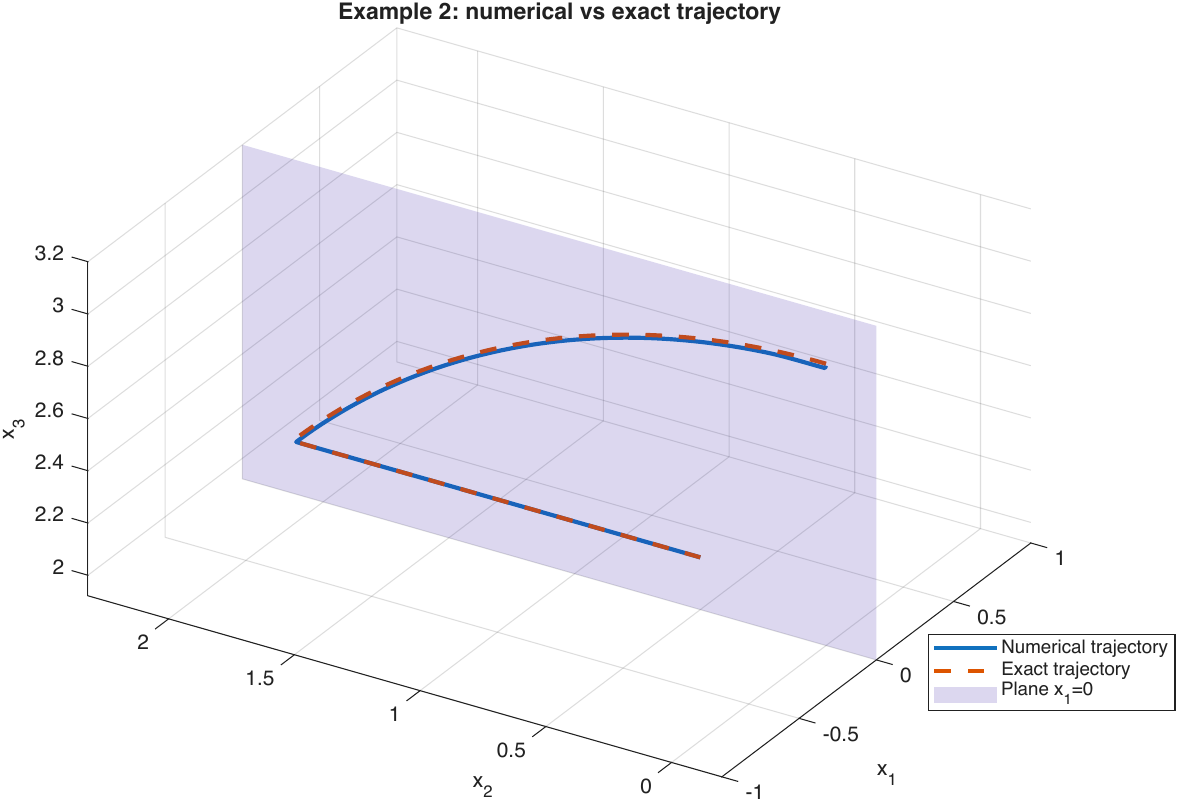}
\includegraphics[width=80mm]{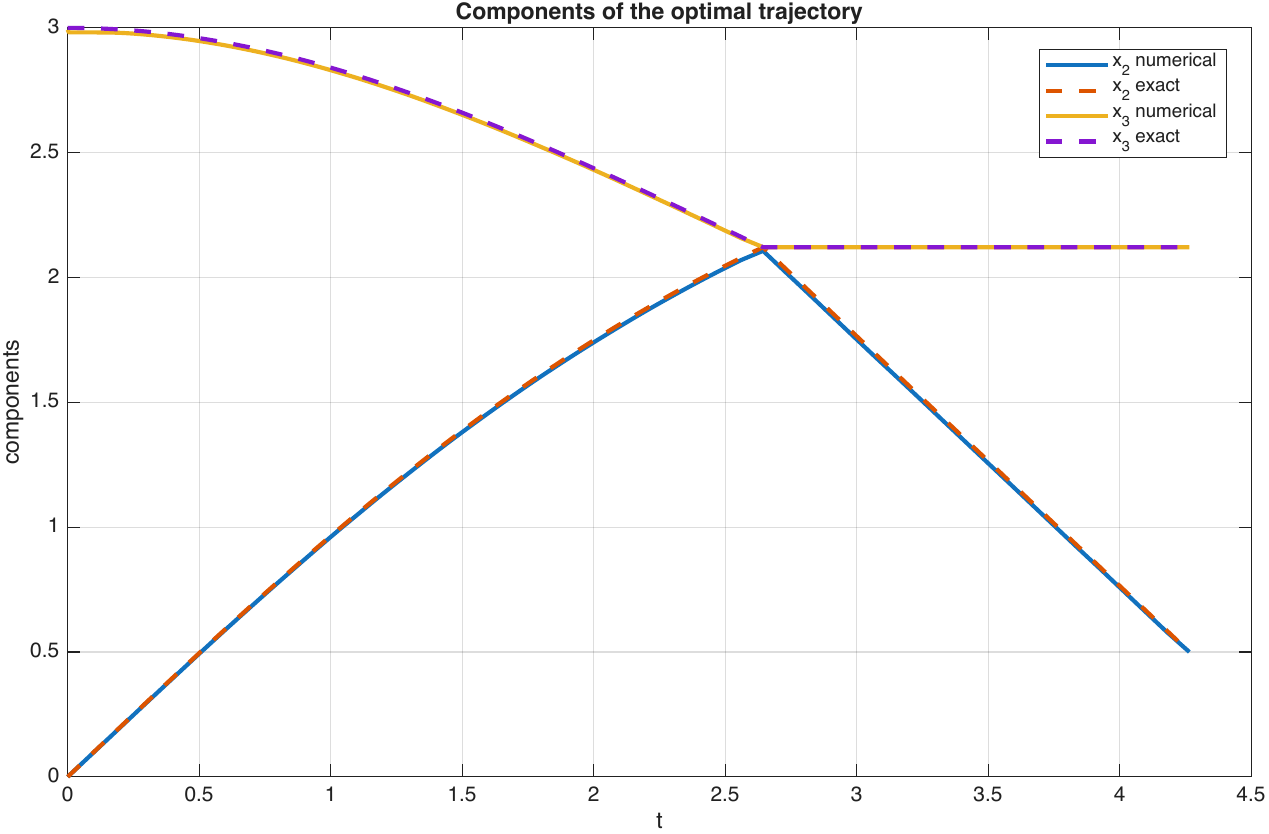}
\caption{\label{Fig4} Numerical vs exact optimal trajectory}
\end{figure}
For this example, we choose $N=50$, $\e=10^{-5}$, $\gamma=20$,\footnote{Since $\frac{2Me}{\eta}\approx 15.28$.} and $\delta=10$, and apply Algorithm \ref{alg1} in order to numerically compute an approximation of both the minimum value and an optimal trajectory associated with $(P)$. Starting from $\gamma_1=20$, the algorithm reached the prescribed tolerance $\e$ after $12$ iterations, corresponding to the value $\gamma_{12}=130$. The obtained numerical cost is $$ g_{12}=g(\hat{x}_{\gamma_{12}}^N(3))=0.000000000000,$$
which is in excellent agreement with the exact minimum value of $(P)$ computed above. The total running time of the algorithm was $112.097$ seconds.\footnote{Machine: MacBook Pro, Apple M3 chip, 16GB Unified Memory.} Moreover, Fig.~\ref{Fig4} shows that the computed numerical optimal trajectory $\hat{x}_{\gamma_k}^N$ accurately reproduces the qualitative behavior of the exact optimal trajectory $\bar x$, including the constrained motion along the nonsmooth part of the boundary, the boundary-to-interior transition motion, and the final evolution inside the sweeping set, while Fig.~\ref{Fig5} illustrates the very rapid convergence of the numerical minimum values toward the exact minimum value $0$ as the penalization parameter $\gamma_k$ increases. Altogether, these numerical results confirm the effectiveness of the proposed approximation method in capturing the complex dynamics generated by moving nonsmooth sweeping sets.
\begin{figure}[h!]
\centering
\includegraphics[width=70mm]{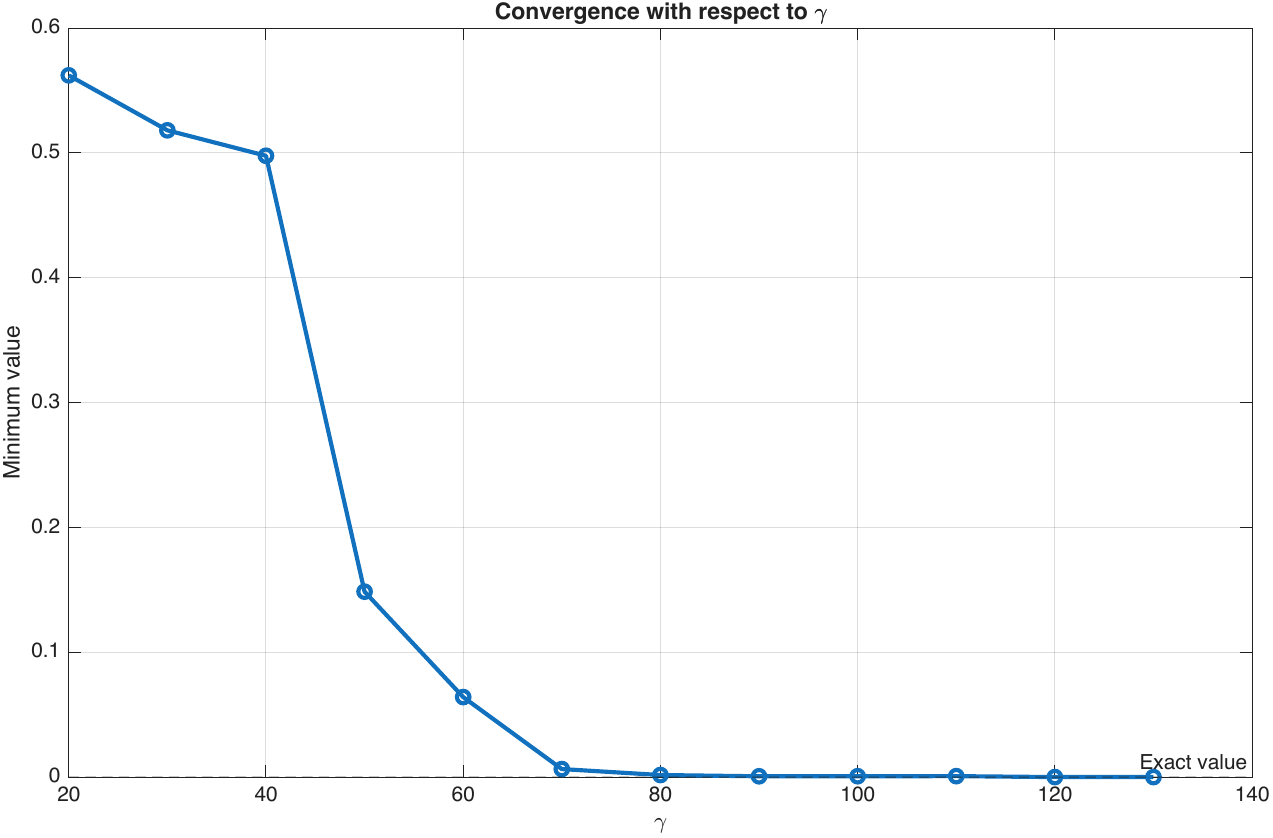}
\caption{\label{Fig5} Convergence with respect to $\gk$}
\end{figure}
\end{example}

\section{Proof of the Main Result} \label{proofsection}

In this section, we present the proof of Theorem \ref{th1}, which is completed at the end of the section. We begin by introducing the following function and families of sets. First, define the function $\psi\colon[0,T]\times\R^n\to\R$
by
\[\psi(t,x):=\max\{\psi_i(t,x): i=1,\dots,r\},\;\, \forall (t,x)\in[0,T]\times\R^n.\]
Clearly, for every $t\in[0,T]$, we have $
C(t)=\{x\in\R^n:\psi(t,x)\le 0\}$. Next, for $i=1,\dots,r$, we define the sequence $(C_i(t,k))_k$ by
\[
C_i(t,k):=\{x\in\R^n:\psi_i(t,x)\le -\a_k\}.
\]
Moreover, we introduce the two families $(C^{\gk}(t))_k$ and $(C^{\gk}(t,k))_k$ defined by
\begin{equation*}\label{Cgkdef}
C^{\gk}(t):=
\{x\in\R^n:\psi_{\gk}(t,x)\le 0\}
=
\bigg\{
x\in\R^n:
\sum_{i=1}^{r}e^{\gk\psi_i(t,x)}
\le 1
\bigg\}, \end{equation*}
and
\begin{eqnarray*}\label{Cgkkdef}
C^{\gk}(t,k)&:=&\nonumber
\{x\in\R^n:\psi_{\gk}(t,x)\le -\a_k\}\\
&=&
\bigg\{
x\in\R^n:
\sum_{i=1}^{r}e^{\gk\psi_i(t,x)}
\le \frac{2M}{\eta\gk}
\bigg\}.
\end{eqnarray*}

\subsection{Preparatory Results}

This subsection is devoted to establishing the preliminary results needed for the proof of Theorem \ref{th1}. Compared with the autonomous framework treated in \cite{verachadinum2}, the main additional difficulty here lies in obtaining estimates that are uniform with respect to the time variable. In particular, several geometric properties of the moving sets $C(t)$, $C_i(t)$, $C_i(t,k)$, $C^{\gk}(t)$, and $C^{\gk}(t,k)$ must be established with constants independent of $t\in[0,T]$. Unlike the autonomous case treated in \cite{verachadinum2}, we restrict ourselves here to the preparatory results that are {\it directly needed} for the proof of the main theorem. Some of the following results are nonautonomous counterparts of preparatory results established in the autonomous setting for the Pontryagin maximum principle in \cite{verachadijca}, and for the numerical approximation method in \cite{verachadinum1,verachadinum2}. Related geometric estimates in the nonautonomous framework were later developed in \cite{verasamara} in connection with the extension of the Pontryagin maximum principle to moving sweeping sets. When the proofs follow by straightforward modifications of the previous arguments, we only indicate the necessary changes. The remaining results require new arguments in order to handle the explicit time dependence of the sweeping set and to obtain estimates that are uniform with respect to time.

We begin with a first collection of uniform geometric properties for the moving sets $C_i(t)$ and $C(t)$. Using Remark \ref{rem1}, and applying, for each $t\in [0,T]$,  \cite[Lemmas 3.3 and 3.4]{VCpaper} for $C_i(t)$ and \cite[Proposition 4.1]{verachadijca} for $C(t)$, we deduce the following proposition (see also \cite{adly} for the prox-regularity of sublevel sets).
\begin{proposition} \label{propo01} The following assertions hold:
\begin{enumerate}[$(i)$, leftmargin=\widthof{(ii)}+\labelsep]
\item For every $t\in[0,T]$, the compact sets $C(t)$ and $C_i(t)$ are both $\frac{\eta}{M_\psi}$-prox-regular with $C(t)=\clo(\inte C(t))$ and $C_i(t)=\clo(\inte C_i(t))$ for $i=1,\dots,r$.
\item For every $t\in[0,T]$, we have: \begin{enumerate}[$\bullet$]
\item $\bdry C_i(t)=\{x\in\R^n : \psi_i(t,x)=0\}\not=\emptyset$ for $i=1,\dots,r$.
\item $\inte C_i(t)=\{x\in\R^n : \psi_i(t,x)<0\}\not=\emptyset$ for $i=1,\dots,r$.
\item $\bdry C(t)=C(t)\cap \left(\bigcup_{i=1}^{r} \bdry C_i(t)\right)\not=\emptyset$.
\item $ \inte C(t)=\bigcap_{i=1}^{r}\inte C_i(t)\not=\emptyset$.
\item For $i=1,\dots,r$ and for all  $x\in\bdry C_i(t)$,  $$N_{C_i(t)}(x)=N^P_{C_i(t)}(x)=N^L_{C_i(t)}(x)=\{\lambda\nabla_x\psi_i(t,x) : \lambda\geq 0\}\not=\{0\}.$$
\item For all $x\in\bdry C(t)$, \begin{eqnarray*}N_{C(t)}(x)=N^P_{C(t)}(x)&=&N^L_{C(t)}(x)\\&=&\left\{\sum_{i\in\I_x^0}\lambda_i\nabla_x\psi_i(t,x) : \l_i\geq 0\right\}\not=\{0\}.\end{eqnarray*}
\end{enumerate}
\end{enumerate}
\end{proposition}

The next lemma plays an important role in the proof of the main theorem. It provides a uniform quasiconvexity property for the boundaries of the moving sets $C_i(t)$. In contrast with the autonomous case, one cannot directly apply \cite[Theorem 5.8$(ii)$]{verachadinum2}, since here it is essential to obtain quasiconvexity constants that are independent of the time variable. This property will be crucial in establishing corresponding uniform quasiconvexity estimates for the approximating sets $C_i(t,k)$ used later in the proof of Theorem \ref{th1}.

\begin{lemma}\label{lem0} Assume that $n>1$. Then, for every $i=1,\ldots,r$, there exists $c_i>0$, independent of $t$, such that $\operatorname{bdry} C_i(t)$ is $c_i$-quasiconvex for every $t\in[0,T]$. More precisely, for every $t\in[0,T]$ and every $x,y\in\operatorname{bdry} C_i(t)$, there exists a $\CO^1$ curve $\Gamma\colon[0,1]\to\operatorname{bdry} C_i(t)$ such that
\[ \Gamma(0)=x,\;\;\Gamma(1)=y,\;\;\hbox{and}\;\;\ell(\Gamma):=\int_0^1\|\dot\Gamma(s)\|\,ds \le c_i\|y-x\|. \]
\end{lemma}

\begin{proof}
Fix $i\in\{1,\ldots,r\}$. From Remark \ref{rem1}, we know that for every $t\in[0,T]$ and every $c\in\bdry C_i(t)$, we have
$$ \|\nabla_x\psi_i(t,c)\|>2\eta.$$
Moreover, from (H3) we also have
$$
\|\nabla_x\psi_i(t,x)-\nabla_x\psi_i(t,y)\|
\le 2M_\psi\|x-y\|,\;\;\forall x,y\in B(0;2R)\;\,\hbox{and}\,\;\forall t\in[0,T].
$$
Hence, for $\rho:=\min\left\{\frac{\eta}{2M_\psi},\frac{R}{2}\right\}$, we have for every $c\in\bdry C_i(t)$ and every
$x\in B(c;\rho)$ that $$
\|\nabla_x\psi_i(t,x)-\nabla_x\psi_i(t,c)\|
\le 2M_\psi\rho\leq \eta.
$$
Consequently, after a rotation sending the unit vector $\frac{\nabla_x\psi_i(t,c)}{\|\nabla_x\psi_i(t,c)\|}$ onto the last coordinate vector, the derivative of $\psi_i(t,\cdot)$ in the last coordinate direction satisfies $ \left|\partial_{x_n}\psi_i(t,x)\right|\ge \eta$ for every $x\in B(c;\rho)$. Therefore, by the implicit function theorem
with parameters uniform in $t$ and $c$, the set
$\bdry C_i(t)\cap B(c;\rho)$ is the graph of a $\CO^{1,1}$ function whose
Lipschitz constant and chart radius depend only on $\eta$ and $M_\psi$,
and not on $t$ or $c$. Thus there exist constants $\rho_0>0$ and $L_0>0,
$ depending only on $\eta$, $M_\psi$ and $R$, such that every point of 
$\bdry C_i(t)$ admits a local representation as a Lipschitz graph on a
ball of radius $\rho_0$, with Lipschitz constant at most $L_0$. In particular, if two points $x,y\in\bdry C_i(t)$ belong to the same
local chart, then they can be joined in $\bdry C_i(t)$ by a Lipschitz
curve whose length is bounded by
$$
\sqrt{1+L_0^2}\,\|x-y\|.
$$
We now pass from this local estimate to a global one, uniformly in
$t$. Since, by (H3),
$$
C_i(t)\subset B(0;R),\;\; \forall t\in[0,T],
$$
we also have
$$
\bdry C_i(t)\subset \overline B(0;R).
$$
Let $K_0:=\sqrt{1+L_0^2}$ and fix $t\in[0,T]$. Take a maximal $\frac{\rho_0}{4}$-separated family $\{a_1,\ldots,a_N\}\subset \bdry C_i(t)$. Then
\[ \bdry C_i(t)\subset \bigcup_{j=1}^N B\left(a_j;\frac{\rho_0}{4}\right),\]
and, by the usual volume comparison argument, $N$ is bounded above by a constant depending only on $n$, $R$ and $\rho_0$.  Indeed, denoting by $\vol$ the Lebesgue measure in $\R^n$, we have $$N\leq  \frac{\vol\left[B\left(0;R+\frac{\rho_0}{8}\right)\right]}{\vol\left[B\left(0;\frac{\rho_0}{8}\right)\right]}.$$
In particular, this bound is independent of $t$. For each $j$, set \[ U_j:=\bdry C_i(t)\cap B\left(a_j;\frac{\rho_0}{2}\right).\]
Since $U_j$ is contained in one of the above local charts, any two points of $U_j$ can be joined in $\bdry C_i(t)$ by a Lipschitz curve of length at most $K_0$ times their Euclidean distance. Since $\bdry C_i(t)$ is connected and $\{U_j\}_{j=1}^N$ is a finite relative open cover of $\bdry C_i(t)$, the intersection graph of this cover is connected. Hence, for any $x,y\in\bdry C_i(t)$, there exist indices $j_1,\ldots,j_m$, with $m\le N$, such that
\[ x\in U_{j_1},\;\;y\in U_{j_m},\;\;\hbox{and}\;\;U_{j_k}\cap U_{j_{k+1}}\neq\emptyset\;\;\text{for }k=1,\ldots,m-1. \]
Choose points \[
z_k\in U_{j_k}\cap U_{j_{k+1}}\;\;\hbox{for}\;\;  k=1,\ldots,m-1. \] Set $z_0:=x$ and $z_m:=y$. Since $z_{k-1},z_k\in U_{j_k}$, they can be joined inside $\bdry C_i(t)$ by a Lipschitz curve of length at most \[ K_0\|z_k-z_{k-1}\|. \] Moreover, because both points belong to $B\left(a_{j_k};\frac{\rho_0}{2}\right)$, we have \[ \|z_k-z_{k-1}\|\le \rho_0. \] Concatenating these curves gives a Lipschitz curve in $\bdry C_i(t)$ joining $x$ to $y$ with length bounded by \[\ell(\gamma)\le mK_0\rho_0\le NK_0\rho_0. \]
Therefore there exists $D_i>0$, independent of $t$, such that every two points of $\bdry C_i(t)$ can be joined in $\bdry C_i(t)$ by a Lipschitz curve of length at most $D_i$. On the other hand, let $x,y\in \bdry C_i(t)$ and suppose that $\|x-y\|<\frac{\rho_0}{2}$. Then $y\in B(x;\rho_0)$, and hence $x$ and $y$ belong to the same local graph representation of $\bdry C_i(t)$, namely the one centered at $x$. Therefore, by the local estimate above, they can be joined by a Lipschitz curve $\Gamma$ in $\bdry C_i(t)$ satisfying
$$\ell(\Gamma)\le K_0\|x-y\|.$$
If $\|x-y\|\ge\frac{\rho_0}{2}$, then the preceding global bound gives that $x$ and $y$ can be joined by a Lipschitz curve $\gamma$  in $\bdry C_i(t)$ satisfying \[ \ell(\Gamma)\le D_i \le \frac{2D_i}{\rho_0}\|x-y\|.\]
Thus every two points $x, y\in \bdry C_i(t)$ can be joined by a Lipschitz curve $\Gamma$  in $\bdry C_i(t)$ satisfying $\ell(\Gamma)\le c_i\|x-y\|$ where \[c_i:=\max\left\{K_0,\frac{2D_i}{\rho_0}\right\}\;\hbox{is independent of}\;t .\] 
Now, each local graph representation of $\bdry C_i(t)$ is a $\CO^{1,1}$ hypersurface. Hence, by smoothing the piecewise Lipschitz parametrization inside the local charts and preserving the endpoints, the Lipschitz curve can be approximated by a $\CO^1$ curve contained in $\bdry C_i(t)$. Possibly after increasing the constant $c_i$, but still independently of $t$, we obtain a $\CO^1$ curve $\Gamma$ satisfying $\ell(\Gamma)\le c_i\|x-y\|$. This proves the result.  \qed \end{proof}

We next establish several uniform geometric properties of the approximating sets $C_i(t,k)$ that will be used repeatedly in the sequel.

\begin{remark}
Since the family $(\psi_i)_{i=1}^r$ is finite, whenever needed, intermediate constants arising in the proofs may be chosen uniformly with respect to $i$ by taking the maximum over $i=1,\dots,r$. For simplicity of exposition, this will be done implicitly in the sequel.
\end{remark}

\begin{proposition}\label{prop02} The following assertions hold: 
\begin{enumerate}[$(i)$, leftmargin=\widthof{(iii)}+\labelsep]
\item For every $t\in [0,T]$ and for every $k$, the set $C_i(t,k)\subset \inte C_i(t)$ is compact for $i=1,\dots,r$.
\item $\displaystyle \bigg(\sup_{t\in[0,T]}\sup_{c\in\bdry C_i(t)} d(c,\bdry C_i(t,k))\bigg)\longrightarrow 0$ as $k\f\infty$, for $i=1,\dots,r$.
\item There exists $k_1\in\N$ such that for all $k\geq k_1$ we have for every $t\in [0,T]$ and every $i\in\{1,\dots,r\}$:
\begin{enumerate}[$\bullet$]
\item $\bdry C_i(t,k)=\{x\in\R^n : \psi_i(t,x)=-\a_k\}\not=\emptyset$.
\item $\inte C_i(t,k)=\{x\in\R^n : \psi(t,x)<-\a_k\}\not=\emptyset$.
\item $C_i(t,k)$ is $\frac{\eta}{4M_\psi} $-prox-regular with $C_i(t,k)=\clo(\inte C_i(t,k))$,  and for all $x\in \bdry C_i(t,k)$, $$N_{C_i(t,k)}(x)=N^P_{C_i(t,k)}(x)=N^L_{C_i(t,k)}(x)=\{\lambda\nabla_x\psi_i(t,x) : \lambda\geq 0\}.$$
\end{enumerate}
\item For $t\in [0,T]$,  $i\in\{1,\dots,r\}$, and $k\in\N$, let $\pi_k^{i,t}$ be the projection map from $\bdry C_i(t)$ onto $\bdry C_i(t,k)$. Then there exists $k_2\geq k_1$ such that for $k\geq k_2$, $\pi_k^{i,t}$ is uniformly bi-Lipschitz with respect to $t$, $k$, and $i$.
\item Assume that $n>1$. Then there exists $\beta>0$ such that for every $k\ge k_2$, $t\in[0,T]$, and $i=1,\dots,r$, $\operatorname{bdry} C_i(t,k)$ is $\beta$-quasiconvex. More precisely, for every $k\ge k_2$, $t\in[0,T]$, $i=1,\dots,r$, and every $x,y\in\operatorname{bdry} C_i(t,k)$, there exists a Lipschitz curve $\Gamma\colon[0,1]\to\operatorname{bdry} C_i(t,k)$ such that
\[ \Gamma(0)=x,\;\;\Gamma(1)=y,\;\;\hbox{and}\;\;
\ell(\Gamma):=\int_0^1\|\dot\Gamma(s)\|\,ds \le \beta\|y-x\|.\]
\end{enumerate}
\end{proposition}
\begin{proof}
Clearly, $(i)$ holds. The assertion $(ii)$ follows by a standard contradiction argument using the continuity of $\psi_i$, the fact that $\alpha_k\searrow0$, and the uniform local graph representation of $\bdry C_i(t)$ established in Lemma \ref{lem1}. The assertion $(iii)$ follows from $(ii)$ and by straightforward adaptations of the corresponding autonomous arguments in \cite{VCpaper,verachadinum2}. The only point requiring attention is the independence of the constant $k_1$ with respect to the time variable. Indeed, as shown in the proof of Lemma \ref{lem0}, the lower bound $\|\nabla_x\psi_i(t,x)\|\ge\eta$ near $\bdry C_i(t)$, together with the uniform $2M_\psi$-Lipschitz continuity of $\nabla_x\psi_i(t,\cdot)$ on $\overline{B}\left(0;\frac{3R}{2}\right)\supset C_i(t)+\frac{R}{2}\overline{B}$, holds uniformly with respect to $t\in[0,T]$. Therefore, all the estimates used in the autonomous proof remain valid with constants independent of $t$, which yields the existence of $k_1$ independent of $t$. We proceed to prove $(iv)$. From $(ii)$, we deduce the existence of $k_2\geq k_1$ such that $$ d(c,\bdry C_i(t,k)) < \frac{\eta}{8M_\psi},\;\,\forall c\in\bdry C_i(t),\;\forall t\in[0,T],\;\forall i\in\{1,\dots,r\}.$$
Hence, since $C_i(t,k)$ is $\frac{\eta}{4M_\psi}$-prox-regular, we obtain by \cite[Theorem 4.8]{prox} that the projection map $\pi_k^{i,t}\colon\bdry C_i(t)\to\bdry C_i(t,k)$ is single-valued and $2$-Lipschitz for every $t\in[0,T]$, $i\in\{1,\dots,r\}$ and for every $\displaystyle k\ge k_2$. We next prove that $\pi_k^{i,t}$ is onto. Let $c\in\bdry C_i(t,k)$ and define \[
\nu_k^{i,t}(c):=
\frac{\nabla_x\psi_i(t,c)}
{\|\nabla_x\psi_i(t,c)\|}. \]
Since $\psi_i(t,c)=-\alpha_k$, $\|\nabla_x\psi_i(t,c)\|\ge\eta$, and $\nabla_x\psi_i(t,\cdot)$ is uniformly $2M_\psi$-Lipschitz on $B(0;2R)$, it follows that the function $s\mapsto\psi_i(t,c+s\nu_k^{i,t}(c))$ is strictly increasing on a uniform interval $[0,\delta]$, where $\delta>0$ is independent of $t$, $k$, $c$, and $i$. Since $\alpha_k\downarrow0$, for $k$ sufficiently large there exists a unique number
$a_k^{i,t}(c)\in(0,\delta)$ such that
\[
\psi_i\bigl(t,c+a_k^{i,t}(c)\nu_k^{i,t}(c)\bigr)=0.
\]
Consequently,
\[
c+a_k^{i,t}(c)\nu_k^{i,t}(c)\in\bdry C_i(t).
\]
Increasing $k_2$ if necessary, we may assume that $a_k^{i,t}(c)<\frac{\eta}{4M_\psi}$ for all $t\in[0,T]$ and all $c\in\bdry C_i(t,k)$. Since $\nabla_x\psi_i(t,c)$ is an outward proximal normal to $C_i(t,k)$ at $c$ and $a_k^{i,t}(c)<\frac{\eta}{4M_\psi}$, the uniqueness of the projection onto $C_i(t,k)$ in its prox-regularity tube yields \[
\pi_k^{i,t}\bigl(c+a_k^{i,t}(c)\nu_k^{i,t}(c)\bigr)=c,
\]
proving that $\pi_k^{i,t}$ is onto. It remains to show that $(\pi_k^{i,t})^{-1}$ is uniformly Lipschitz. Applying the implicit function theorem to the equation
\[
\psi_i\bigl(t,c+a\nu_k^{i,t}(c)\bigr)=0,
\]
we deduce, using the uniform lower bound on the derivative with respect to $a$,
that the map $c\mapsto a_k^{i,t}(c)$ is Lipschitz on $\bdry C_i(t,k)$, with Lipschitz constant independent of $t$, $k$, and $i$. Since the map
$
c\mapsto \nu_k^{i,t}(c)
$
is also uniformly Lipschitz, it follows that
\[
(\pi_k^{i,t})^{-1}(c)
=
c+a_k^{i,t}(c)\nu_k^{i,t}(c)
\]
is uniformly Lipschitz on $\bdry C_i(t,k)$. Therefore, $\pi_k^{i,t}$ is uniformly bi-Lipschitz for all $k\ge k_2$. Finally, $(v)$ follows directly from Lemma \ref{lem0} and the uniform bi-Lipschitz property established in $(iv)$.  \qed  \end{proof}

We next establish several elementary properties of the families $(C^{\gk}(t))_k$ and $(C^{\gk}(t,k))_k$. In contrast with the previous results, the proofs here follow by straightforward adaptations of the corresponding autonomous arguments in \cite{verachadinum2}, since no additional geometric estimates depending on the time variable are required.

\begin{proposition}\label{prop03}
The following assertions hold\sp$:$
\begin{enumerate}[$(i)$, leftmargin=\widthof{(iii)}+\labelsep]
\item For every $t\in [0,T]$ and every $k$,  we have  \begin{equation} \label{cgk(k)toci(k)} C^{\gk}(t,k)\subset \inte C^{\gk}(t)\subset \inte C(t)\;\;\hbox{and}\;\; C^{\gk}(t,k)\subset \bigcap_{i=1}^{r} C_i(t,k).\end{equation}
\item For every $t\in[0,T]$, the sequence $(C^{\gk}(t,k))_k$ is a nondecreasing sequence whose  Painlev\'e-Kuratowski limit is $C(t)$ and satisfies
\begin{equation}\label{union.Cgk(k)}\inte C(t)=\bigcup_{k\in \N} \inte C^{\gk}(t,k)=\bigcup_{k\in \N} C^{\gk}(t,k). \end{equation}
\item For $c\in \bdry C(0)$, there exist ${k}_c\geq k_2$ such that \begin{equation*}\label{boundarypoint(k)}
 \left({c}+\sigma_k\frac{d_c}{\|d_c\|}\right)\in \inte C^{\gk}(0,k)\subset\inte C^{\gk}(0). \end{equation*}
\end{enumerate}
\end{proposition}

For the sequence $(x_0^k)_k$ associated with $x_0$ and defined in \eqref{initial}, the following lemma follows from the fact that $\sigma_k\f0$, together with \eqref{union.Cgk(k)} and Proposition \ref{prop03}$(iii)$.

\begin{lemma}\label{lem1} The sequence $x_0^k$ converges to $x_0$, and there exists $k_3\geq k_{x_0}$ such that $x_0^k\in C^{\gk}(k)$ for all $k\geq k_3$.
\end{lemma}

\subsection{Key Results} 
Parallel to \cite[Subsection 5.2]{verachadinum2}, we establish in this subsection the main auxiliary results underlying our numerical approximation method for $(P)$. Compared with the autonomous framework treated in \cite{verachadinum2}, the explicit time dependence of the sweeping set creates substantial additional difficulties, since several geometric and analytical quantities may a priori depend on the time variable. Therefore, particular attention will be devoted to obtaining estimates that are uniform with respect to $t\in[0,T]$.

Let $u\in\mathscr{U}$, and let $x_{\gk}$ be the solution of $(D_{\gk})$, given in \eqref{Dgk1}, corresponding to the initial point $x_0^k$ defined in \eqref{initial}. Associated with $x_{\gk}$, we introduce the nonnegative continuous function $\xi_{\gk}\colon[0,T]\to\R$ defined by
\begin{equation*}
\xi_{\gk}(t):=
\gk e^{\gk\psi_{\gk}(t,x_{\gk}(t))}
=
\sum_{i=1}^{r}\xi_{\gk}^i(t),\;\,\forall t\in[0,T],
\end{equation*}
where, for $i=1,\dots,r$,
\[
\xi_{\gk}^i(t):=
\gk e^{\gk\psi_i(t,x_{\gk}(t))},\;\,\forall t\in[0,T].
\]

The following proposition, which coincides with \cite[Proposition 5.6]{verachadinum2}, follows, like the latter, along the same lines as \cite[Theorem 4.13]{verachadijca}. The main additional point in the present nonautonomous setting is to verify that all estimates used in that argument, including those ensuring the invariance of the penalized trajectories in the moving inner approximations, can be chosen uniformly with respect to $t\in[0,T]$. This follows from the uniform assumptions in (H3), in particular from the uniform regularity of the defining functions $\psi_i$ with respect to both $t$ and $x$. In the invariance argument, the time dependence produces an additional term controlled by the constant $M_t$, which is absorbed into the constant $M$ according to Remark \ref{Mt}. More details in a more general setting can be found in \cite[Section 3.2]{verasamara}.

\begin{proposition}  \label{propo1}
There exists $k_4\geq k_3$ such that for all $k\geq k_4$ and for all $u\in \mathscr{U}$, the solution $x_{\gk}$ of $(D_{\gk})$ corresponding to $u$ satisfies:
\begin{enumerate}[$(i)$, leftmargin=\widthof{(iii)}+\labelsep]
\item $x_{\gk}(t)\in C^{\gk}(t,k)\subset \inte C^{\gk}(t)\subset \inte C(t)$ for all $t\in[0,T]$.
\item $0< \xi_{\gk}^i(t)\le \xi_{\gk}(t)\le \frac{2M}{\eta}$ for all $t\in[0,T]$ and for $i=1,\dots,r.$
\item $\|\dot{x}_{\gk}(t)\|\le M+\frac{2M\bar{M}_\psi}{\eta}$ \,for \textnormal{a.e.} $t\in[0,T]$.
\end{enumerate}
\end{proposition}

We now generalize \cite[Proposition 5.7]{verachadinum2} to the present nonautonomous setting with moving sweeping sets.

\begin{proposition}  \label{propo2} 
For all $k\geq k_4$ and for all $u\in \mathscr{U}$, the solution $x$ of system $(D)$ and the solution $x_{\gk}$ of system $(D_{\gk})$, both corresponding to the same control $u\in\U$, satisfy 
\begin{equation*}\label{estimate} \|x_{\gk}(t)-x(t)\|^2\leq e^{\tilde{M}T} \sigma_k^2 + \frac{8\eta M (e^{\tilde{M}T}-1)}{\tilde{M}M_{\psi}}\sigma_k,\;\,\forall t\in [0,T], \end{equation*}
where $\tilde{M}:=\frac{5MM_\psi}{ \eta}+2{M}$.
\end{proposition}
\begin{proof}
The proof is identical to that of Proposition 5.7 in \cite{verachadinum2}. Although the sweeping set depends explicitly on time, the additional difficulties can be overcome by means of \eqref{H3ineq1} together with the uniform prox-regularity of $C(t)$, $t\in[0,T]$, proved in Proposition \ref{propo01}. \qed
\end{proof}

We next extend \cite[Lemma 5.2]{verachadinum2} to the present nonautonomous setting with moving sweeping sets. In contrast with the autonomous case, a crucial point here is to obtain quasiconvexity estimates for the boundaries of the approximating sets $C_i(t,k)$ that are uniform with respect to the time variable. This is precisely ensured by Proposition \ref{prop02}$(v)$ and plays a fundamental role in the proof of the next lemma.

\begin{lemma} \label{lem2} Let $n>1$ and let $\b$ be the constant of \textnormal{Proposition \ref{prop02}$(v)$}. Then for every $k\geq k_4$, $t\in[0,T]$, and every $x,y\in C^{\gk}(t,k)$, we have $$\<\nabla_x e^{\gk\psi_{\gk}(t,y)}- \nabla_x e^{\gk \psi_{\gk}(t,x)},y-x\>\ge -\frac{4r\beta M M_{\psi}}{\eta}\|y -x\|^2.$$
\end{lemma}
\begin{proof}
We assume, without loss of generality, that $\beta\geq 1$. Fix $k\ge k_4$, $t\in [0,T]$, $i\in\{0,\dots,1\}$, and $x,y \in C_i(t,k)$, with $x\neq y$. Define
$$\textstyle  G_i(z):=\<\nabla_x e^{\gk \psi_i(t,z)}, y-x\>,\;\;\forall z\in\R^n.$$
Since $\psi_i(t,\cdot)$ is $\CO^{1,1}$, the {\it nonsmooth} product and chain rules (see \cite{clarkeold}) imply that, for $z\in \R^n$,  
\begin{equation} \textstyle\label{hessianepsi}\partial^2_xe^{\gk \psi_i(t,z)}\subset \gk e^{\gk \psi_i(t,z)}[\partial^2_x\psi_i(t,z)+\gk\nabla_x\psi_i(t,z)\otimes\nabla_x\psi_i(t,z)],
\end{equation}
where the tensor $\nabla_x\psi_i(t,z)\otimes\nabla_x\psi_i(t,z)$ is positive semi-definite, and, for any $$z\in C_i(t,k)\subset \inte C_i(t)\subset C(t)\subset B(0;R),$$ and $M_z\in\partial^2_x\psi_i(t,z)$, we have, by (H3), $\|M_z\|\le 2M_{\psi}$, and $$0<\gk e^{\gk \psi_i(t,z)}\leq \gk e^{-\gk\a_k}\leq\frac{2M}{\eta}.$$
\underline{Case 1:} $[x,y]\subset C_i(t,k)$.\vspace{0.1cm}\\
By Lebourg's mean value theorem, see \cite[Theorem 2.3.7]{clarkeold}, and \eqref{hessianepsi}, there exist $z\in]x,y[\subset C^{\gk}(t,k)$ and $\mathcal{M}_z\in\partial^2_xe^{\gk \psi_i(t,z)}$ such that 
\begin{eqnarray*}\nonumber  \<\nabla_x e^{\gk\psi_i(t,y)}- \nabla_x e^{\gk \psi_i(t,x)},y-x\>&=&\<\mathcal{M}_z(y-x),y-x\>\\&\ge&  -\frac{4M M_{\psi}}{\eta}\|y -x\|^2 \\&\geq&  -\frac{4\beta M M_{\psi}}{\eta}\|y -x\|^2.\label{gradepsi1}
\end{eqnarray*} 
\underline{Case 2:} $]x, y[ \not\subset C_i(t,k)$ and $x,\sp y\in\bdry C_i(t,k)$. \vspace{0.1cm}\\ 
By Proposition \ref{prop02}$(v)$, there exists a Lipschitz curve having length $\ell=\ell(\Gamma)$ with $\Gamma\colon [0,1]\to \bdry C_i(t,k)$ satisfying $\Gamma(0)=x$, $\Gamma(1)=y$, $\ell\le \beta \|x-y\|$, and
\begin{equation} \label{numnew1} \int_0^1\|\dot{\Gamma}(s)\|\sp ds=\ell\;\,\hbox{and}\;\,\<\nabla_x\psi_i(t,{\Gamma}(s)),\dot{\Gamma}(s)\>=0,\;\hbox{a.e.}\;s.\end{equation}
By the nonsmooth chain rule, (\ref{hessianepsi}), and (\ref{numnew1}), there exists $M_{{\Gamma}}\in L^{\infty}$ with $M_{{\Gamma}}(s)\in \partial^2_x\psi_i(t,{\Gamma}(s))$ a.e. such that  
 \begin{eqnarray*} \nonumber
\<\nabla_x e^{\gk\psi_i(t,y)}- \nabla_x e^{\gk \psi_i(t,x)},y-x\>&=&G_i(\Gamma(1))-G_i(\Gamma(0))\\&=&  \nonumber \int^1_0 \frac{dG(\Gamma(s))}{ds}\sp ds \\&= &\int_0^1\gk e^{\gk\psi_i(t,{\Gamma}(s))}\<M_{{\Gamma}}(s)\dot{{\Gamma}}(s),y-x\>\sp ds
 \nonumber\\&\geq&\frac{-4MM_\psi}{\eta}\|x-y\|\int_0^1\|\dot{\Gamma}(s)\|\sp ds\\&\geq&   -\frac{4\beta M M_{\psi}}{\eta}\|y -x\|^2.
\end{eqnarray*}
\underline{Case 3:} $]x, y[ \not\subset C_i(t,k)$ and $x$ or $y\not\in\bdry C_i(t,k)$. \vspace{0.1cm}\\  
Set $\x$ and $\y$  to be, respectively, the first and the last points of $[x,y]$ in $\bdry C_i(t,k)$. Now, by applying  Case 1 to $[x,\x]$ and $[\y,y]$, and Case 2  to $[\x,\y]$, and by noticing that $\x-x=\l_1(y-x)$, $\y-\x=\l_2(y-x)$ and $y-\y=\l_3(y-x)$, where $\l_i\geq 0$ for all $i\in\{1,2,3\}$  and $\sum_{i=1}^{3}\l_i=1$, we deduce that $$\<\nabla_x e^{\gk\psi_i(t,y)}- \nabla_x e^{\gk \psi_i(t,x)},y-x\> \geq   -\frac{4\beta M M_{\psi}}{\eta}\|y -x\|^2.$$
Therefore, for every $k\geq k_4$, $t\in [0,T]$, and every $i\in\{0,\dots,1\}$, we have
$$\<\nabla_x e^{\gk\psi_i(t,y)}- \nabla_x e^{\gk \psi_i(t,x)},y-x\>\ge -\frac{4\beta_i M M_{\psi}}{\eta}\|y -x\|^2,\;\,\forall x,\sp y\in C_i(t,k).$$
Using that $e^{\gk\psi_{\gk}(\cdot,\cdot)}=\sum_{i=1}^re^{\gk\psi_i(\cdot,\cdot)}$ and the inclusion \eqref{cgk(k)toci(k)}, we conclude that for evrey $k\geq k_4$, $t\in[0,T]$, and every $x,y\in C^{\gk}(t,k)$, we have $$\<\nabla_x e^{\gk\psi_{\gk}(t,y)}- \nabla_x e^{\gk \psi_{\gk}(t,x)},y-x\>\ge -\frac{4r\beta M M_{\psi}}{\eta}\|y -x\|^2.$$
This terminates the proof of the lemma.\qed \end{proof}

\begin{remark}\label{convex}
Assume that for every $i=1,\dots,r$ and every $t\in[0,T]$, the set $C_i(t)$ is convex (equivalently, $\psi_i(t,\cdot)$ is convex). Then Lemma \ref{lem2} becomes immediate. Indeed, in this case the function $\psi_{\gk}(t,\cdot)$ is convex for every $t\in[0,T]$, see \cite[Lemma 3]{LiFang}. Consequently, the function $e^{\gk\psi_{\gk}(t,\cdot)}$ is convex, and therefore
$$\<\nabla_x e^{\gk\psi_{\gk}(t,y)}-\nabla_x e^{\gk \psi_{\gk}(t,x)},y-x\>\geq 0,
\;\,\forall x,y\in\R^n.$$
Moreover, when $n>1$, the convexity of $C_i(t)$ implies that $\bdry C_i(t)$ is connected.
\end{remark}

As in \cite{verachadinum1,verachadinum2}, Lemma \ref{lem2} yields the following proposition, which constitutes a crucial step in the construction of the discrete approximations. Its proof follows arguments similar to those used in the proofs of \cite[Proposition 3]{verachadinum1} and \cite[Proposition 5.8]{verachadinum2}.

\begin{proposition} \label{propo3} Let $u\in \mathscr{U}$ and, for $k\ge {k}_4$, let $x_{\gk}$ be the solution of $(D_{\gk})$ corresponding to $u$. Then, for $N\in \N$, there exists  $u^N\in\U^N$ such that $x^N_{\gk}$, the solution of $(D_{\gk})$ corresponding 
to  $u^N,$ satisfies for $\delta_j:=\|x_{\gk}(jh)-x^N_{\gk}(jh)\|^2$ the inequality  
\begin{equation*} \label{xN-xgkN}
\delta_j\leq \frac{2\hat{M}^2 e^{6\bar{M}T}}{3\bar{M}}  (1+6 \bar{M} h)(1+\bar{M}h)h,\,\;\;\hbox{for}\;\,j=1,\dots,N,\end{equation*}
where $\bar{M}:=M+\frac{4r\beta MM_{\psi}}{\eta}$,  $\hat{M}:={M}+\frac{2{M}\bar{M}_\psi}{\eta}$, and $h:=\frac{T}{N}$. \end{proposition}

\subsection{Proof of Theorem \ref{th1}}

The proof of Theorem \ref{th1} now follows exactly along the same lines as the proof of \cite[Theorem 3.1]{verachadinum2}. Indeed, Propositions \ref{propo1}, \ref{propo2}, and \ref{propo3} provide the corresponding nonautonomous counterparts of \cite[Propositions 5.6, 5.7, and 5.8]{verachadinum2}, respectively. Therefore, repeating the same arguments yields the conclusion.

\section{Conclusions}\label{conclusionsection}

In this paper, we extended the numerical approximation method developed in \cite{pinhonum,verachadinum1,verachadinum2} to controlled sweeping processes with moving nonsmooth sweeping sets. The main difficulty comes from the explicit time dependence of the sweeping set, which requires establishing geometric and analytical estimates that are uniform with respect to time. Under suitable assumptions, we proved the convergence of the proposed approximation scheme toward admissible solutions of the original problem.

The numerical simulations presented in Section~\ref{examples} illustrate the effectiveness of the method and show an excellent agreement between the exact and numerical optimal trajectories. Possible extensions of the present work to broader classes of optimal control problems governed by sweeping processes, including free time problems, will be the subject of future research.



\begin{thebibliography}{plain}

\bibitem{outrata} Adam, L.,  Outrata, J.: On optimal control of a sweeping process coupled with an ordinary differential equation. Discrete Contin. Dyn. Syst. B 19, 2709–2738 (2014)

\bibitem{adly} Adly, S.,  Nacry, F.,  Thibault, L.: Preservation of Prox-Regularity of Sets with Applications to Constrained Optimization. SIAM Journal on Optimization. 26:1, 448–473 (2016) 

\bibitem{Bettiol} Bettiol, P., Colombo, G.,  Gidoni, P. : Penalization and Necessary Optimality Conditions for a Class of Nonsmooth Sweeping Processes. Set-Valued Var. Anal. 33, 33 (2025).

\bibitem{brokate} Brokate, M., Krej\v{c}\'{\i}, P.: Optimal control of ODE systems involving a rate independent variational inequality. Discrete and continuous dynamical systems series B 18, 331–348 (2013)

\bibitem{brudnyi} Brudnyi, A., Brudnyi, Y.: Methods of Geometric Analysis in Extension and Trace Problems, Volume 1, Monographs in Mathematics, 102. Birkh\"{a}user/Springer Basel AG, Basel (2012)

\bibitem{ccmn} Cao, T.H., Colombo, G., Mordukhovich, B., Nguyen, D.: Optimization of Fully Controlled Sweeping Processes. J. Differ. Equ. 295, 138–186 (2021)

\bibitem{ccmnbis} Cao, T.H., Colombo, G., Mordukhovich, B., Nguyen, D.: Optimization and discrete approximation of sweeping processes with controlled moving sets and perturbations, J. Differ. Equ. 274, 461–509 (2021)

\bibitem{cmo0} Cao, T.H., Mordukhovich, B.: Optimal control of a perturbed sweeping process via discrete approximations. Discrete Contin. Dyn. Syst. Ser. B 21, 3331–3358 (2016) 

\bibitem{cmo} Cao, T.H., Mordukhovich, B.: Optimality conditions for a controlled sweeping process with applications to the crowd motion model. Disc. Cont. Dyn. Syst. Ser. B 22, 267–306 (2017)

\bibitem{cmo2} Cao, T.H., Mordukhovich, B.: Optimal control of a nonconvex perturbed sweeping process. J. Differ. Equ. 266. 1003–1050 (2019)

\bibitem{cmonew} Cao, T.H., Mordukhovich, B., Nguyen, D., Nguyen, T.,  Thieu, N. N.: Optimal Control of Nonconvex Sweeping Processes with Variable Time via Finite-Difference Approximations. https://arxiv.org/abs/2503.00667v2

\bibitem{verasamara} Chamoun, S., Zeidan, V.: Optimal Control for Coupled Sweeping Processes Under Minimal Assumptions Appl. Math. Optim. 92, 1 (2025).

\bibitem{clarkeold} Clarke, F.H.: Optimization and Nonsmooth Analysis, John Wiley, New York (1983)

\bibitem{prox} Clarke, F.H., Stern, R.J., Wolenski, P.R.: Proximal smoothness and the lower-$C^2$ property. J. Convex
Anal. 2, 117–144 (1995)

\bibitem{clsw} Clarke, F.H.,  Ledyaev, Yu., Stern, R.J., Wolenski, P.R: Nonsmooth Analysis and Control Theory. Graduate Texts in Mathematics, 178, Springer-Verlag, New York (1998)

\bibitem{chhm2} Colombo, G., Henrion, R., Hoang, N.D., Mordukhovich, B.S.: Optimal control of the sweeping process. Dyn. Contin. Discrete Impuls. Syst. Ser. B 19, 117–159 (2012)

\bibitem{chhm} Colombo, G., Henrion, R., Hoang, N.D., Mordukhovich, B.S.: Optimal control of the sweeping process over polyhedral controlled sets. J. Differ. Equ. 260(4), 3407–3447 (2016)

\bibitem{cmn0} Colombo, G., Mordukhovich, B., Nguyen, D.: Optimization of a perturbed sweeping process by constrained discontinuous controls. SIAM J. Control Optim. 58:4, 2678–2709 (2020)

 \bibitem{delfour}  Delfour, M.C.,  Zolésio, J.-P.: Shapes and Geometries, Metrics, Analysis, Differential Calculus, and Optimization. Second edition, Advances in Design and Control, vol. 22, Society for Industrial and Applied Mathematics (SIAM),
Philadelphia PA (2011)

\bibitem{pinho} de Pinho, M.d.R., Ferreira, M.M.A., Smirnov, G.V.: Optimal Control Involving Sweeping Processes. Set-Valued Var. Anal. 27, no. 2, 523–548 (2019)

\bibitem{pinhoEr} de Pinho, M.d.R., Ferreira, M.M.A., Smirnov, G.V.: Correction to: Optimal Control Involving Sweeping Processes. Set-Valued Var. Anal. 27, 1025–1027 (2019)

\bibitem{pinhonum} de Pinho, M.d.R., Ferreira, M.M.A., Smirnov, G.V.: Optimal Control with Sweeping Processes: Numerical Method. J Optim Theory Appl. 185, 845–858 (2020) 

\bibitem{pinholast} de Pinho, M.d.R., Ferreira, M.M.A., Smirnov, G.V.: Necessary conditions for optimal control problems with sweeping systems and end point constraints. Optimization 71:11, 3363–3381 (2021)

\bibitem{pinho22} de Pinho, M.d.R., Ferreira, M.M.A.,  Smirnov, G. A: Maximum Principle for Optimal Control Problems Involving Sweeping Processes with a Nonsmooth Set. J Optim Theory Appl. 199, 273–297 (2023)

\bibitem{henrion} Henrion, R., Jourani, A., Mordukhovich, B.S: Controlled polyhedral sweeping processes: Existence, stability, and optimality conditions, J. Differ. Equ.  366, 408–443 (2023)

\bibitem{palladino} Hermosilla, C., Palladino, M.: Optimal Control of the Sweeping Process with a Nonsmooth Moving Set, SIAM J. Control Optim. 60:5, 2811-2834 (2022)

\bibitem{LiFang} Li, X.-S.,  Fang, S.-C.: On the entropic regularization method for solving min-max problems with applications, Math. Methods Oper. Res. 46, 119–130 (1997)

\bibitem{mordubook} Mordukhovich, B.S.: Variational Analysis and Generalized Differentiation, I: Basic Theory, Springer, Berlin (2006)

\bibitem{moreau1} Moreau, J.J.: Rafle par un convexe variable, I, Trav. Semin. d'Anal. Convexe, Montpellier 1, Expos\'e 15, 36 pp. (1971) 

\bibitem{moreau2} Moreau, J.J.: Rafle par un convexe variable, II, Trav. Semin. d'Anal. Convexe, Montpellier 2, Expos\'e 3, 43 pp. (1972)

\bibitem{moreau3} Moreau, J.J.: Evolution problem associated with a moving convex set in a Hilbert space, J. Differ. Equ. 26, 347-374 (1977)

\bibitem{VCpaper} Nour, C., Zeidan, V.: Optimal control of nonconvex sweeping processes with separable endpoints: Nonsmooth maximum principle for local minimizers. J. Differ. Equ. 318, 113–168 (2022)

\bibitem{verachadinum1} Nour, C., Zeidan, V.: Numerical solution for a controlled nonconvex sweeping process. IEEE Control Syst. Lett. 6, 1190–1195 (2022)

\bibitem{verachadisvaa} Nour, C., Zeidan, V.:  A Control Space Ensuring the Strong Convergence of Continuous Approximation for a Controlled Sweeping Process. Set-Valued Var. Anal. 31, 23 (2023)

\bibitem{verachadijune} Nour, C., Zeidan, V.:  Nonsmooth optimality criterion for a W$^{1,2}$-controlled sweeping process: Nonautonomous perturbation. Appl. Set-Valued Anal. Optim. 5(2), 193–212 (2023)

\bibitem{verachadijca} Nour, C., Zeidan, V.: Pontryagin-Type Maximum Principle for a Controlled Sweeping Process with Nonsmooth and Unbounded Sweeping Set. J. Convex Anal. 31, (2024), to appear.

\bibitem{verachadinum2} Nour, C., Zeidan, V.: Numerical Method for a Controlled Sweeping Process with Nonsmooth Sweeping Set J Optim Theory Appl. 203, 1385--1412, (2024)

\bibitem{rockwet} Rockafellar, R.T., Wets, R.J.-B.: Variational analysis. Grundlehren der Mathematischen Wissenschaften, 317, Springer-Verlag, Berlin (1998)

\bibitem{ThibaultBook}  Thibault, L.: Unilateral Variational Analysis in Banach Spaces. World Scientific (2023)

\bibitem{verachadi} Zeidan, V., Nour, C., Saoud, H.: A nonsmooth maximum principle for a controlled nonconvex sweeping process. J. Differ. Equ. 269(11), 9531–9582 (2021)


\end{thebibliography}
\end{document}